\documentclass[11pt]{amsart}
\usepackage{graphicx}
\usepackage{amssymb, amsmath, amsthm, amscd}
\usepackage{epstopdf}
\RequirePackage[usenames,dvipsnames]{color}

\def\sc#1{{\mathcal #1}}
\def \C  {\sc C}
\def \D {\sc D}
\def \Cf {{\sc C}_f}
\def \Cg {{\sc C}_g}
\def \s  {\sc S}

\def\Q {\mathbb Q}
\def \CQ {{\mathrm{Comm}}_{\Q}}
\def\U {\sc U}

\def\AC { {\empty}_A\backslash \C }

\def\tower { \operatorname{\underline{\bf{Tot}}}}
\def\diag{ \operatorname{diag}}

\def\holim {\operatorname{holim}}
\def\hofib {\operatorname{hofib}}
\def \Tot {\operatorname{Tot}}

\def\sk{\operatorname{sk}}
\def\csk{\operatorname{csk}}
\def\hom{\operatorname{hom}}

\def\tot{\operatorname{tot}}

\input xy
\xyoption{all}
\xyoption{frame}

\theoremstyle{definition}
\newtheorem{Observation}{Observation}

\newtheorem{lem}{Lemma}[section]
\newtheorem{rem}[lem]{Remark}
\newtheorem{prop}[lem]{Proposition}
\newtheorem{thm}[lem]{Theorem}
\newtheorem{cor}[lem]{Corollary}

\newtheorem{defn}[lem]{Definition}
\newtheorem{ex}[lem]{Example}

\title{Combinatorial models for Taylor polynomials of functors}
\author{K. Bauer}
\address{Department of Mathematics \& Statistics, University of Calgary}
\email{bauerk@ucalgary.ca}

\author{R. Eldred}
\address{Mathematisches Institut, Universit\"at M\"unster}
\email{eldred@uni-muenster.de}

\author{B. Johnson}
\address{Department of Mathematics, Union College}
\email{johnsonb@union.edu}

\author{R. McCarthy}
\address{Department of Mathematics, University of Illinois at Urbana-Champaign}
\email{randy@math.uiuc.edu}
\begin{document}

\begin{abstract}

%We use the Taylor towers of \cite{BJM} and \cite{G3} to define a new tower, the varying center tower, for  functors of under categories, that is, categories with a fixed initial object.  We provide a combinatorial model for the limit of this tower.   This is used to show that two models
 %for the Quillen derived de Rham cohomology of rational commutative algebras remain equivalent when extended to $E_\infty$-ring spectra.
 
 Goodwillie's calculus of homotopy functors associates a tower of polynomial approximations, the Taylor tower, to a functor of topological spaces over a fixed space.  We define a new tower, the varying center tower, for  functors of  categories with a fixed initial object, such as algebras under a fixed ring spectrum.  We construct this new tower using elements of the Taylor tower constructions of Bauer, Johnson, and McCarthy for functors of simplicial model categories, and show how the varying center tower differs from Taylor towers in terms of the properties of its individual terms and convergence behavior.    
 We prove that there is a combinatorial model for the varying center tower given as a pro-equivalence between the varying center tower and towers of cosimplicial objects; this generalizes Eldred's cosimplicial models for finite stages of Taylor towers.
 %
% We prove that the limit of the varying center tower is equivalent to a combinatorial construction involving the standard cosimplicial simplicial set $\Delta^{\bullet}_*$.  
 As an application, we present models for the de Rham complex of rational commutative ring spectra due to Rezk on the one hand, and Goodwillie and Waldhausen on the other, and use our result to conclude that these two models will be equivalent when extended to $E_{\infty}$-ring spectra.      \end{abstract}
 
 %%%

\maketitle

\section{Introduction}

This paper arose from a desire to understand the connection between
two seemingly disparate ideas for defining the Quillen derived de Rham cohomology for 
rational commutative algebras and their extensions to ring spectra.   The first, suggested by Friedhelm Waldhausen and Tom Goodwillie, uses Goodwillie's 
calculus of homotopy functors, a theory that has played a significant role in understanding
connections between $K$-theory, Hochschild homology, and related constructions.
The second  is a combinatorial approach explored by Charles Rezk.

For a homotopy functor of spaces or spectra,  that is, a functor that  preserves weak equivalences, Goodwillie constructed  a tower of functors and natural transformations that can be viewed as playing the role of a Taylor series for the functor \cite{G3}.  This construction has been extended to more general contexts, such as simplicial or topological model categories, in \cite{K} and \cite{BJM}.   In this paper, 
we work primarily with the discrete calculus,  a variant of Goodwillie calculus developed in \cite{BJM}.   The discrete calculus associates to a functor $F:\C\rightarrow \s$ and morphism $f:A\rightarrow B$ in $\C$,  a sequence of functors $\Gamma_n^fF:\Cf \to \s$ and 
natural transformations 
\begin{equation}\label{e:taylortower} \xymatrix{ && F \ar[dl]_{\gamma_{n+1}^f}\ar[d]^{\gamma_n^f}\ar[dr]^{\gamma_{n-1}^f}\quad \\
\cdots \ar[r] & \Gamma_{n+1}^fF \ar[r]^{q_{n+1}^f} & \Gamma_n^fF \ar[r]^{q_n^f\quad} & \Gamma_{n-1}^fF \ar[r] ^ \cdots \ar[r] & \Gamma_1^fF \ar[r] & \Gamma_0^fF}\end{equation}
where $\C$ is a simplicial model category,  the category $\Cf$ is the category of factorizations of $f:A\to B$ in $\C$, and $\s$ is a suitable model category of spectra, such as that of \cite{EKMM} or \cite{HSS}.
The $n$th term in this tower can be thought of  as a degree $n$ approximation to $F$, where being degree $n$ means that  the functor takes certain types of $(n+1)$-cubical homotopy pushout diagrams to homotopy pullback diagrams.   Goodwillie's original formulation of  polynomial degree $n$ functors used a stronger notion of degree $n$.
 If $F$ is  nice (in the sense of Definition \ref{d:rho}), the discrete tower converges to $F$.   That is, $F$ is weakly equivalent to the homotopy inverse limit, $\Gamma^f_{\infty}F$, of the tower.

To explain how we see the de Rham complex in terms of Taylor towers, we  first look at an analogous situation for the Taylor series of a function.    Consider a function $f:{\mathbb R}\rightarrow {\mathbb R}$ and a real number $b$. We can construct the $n$th Taylor polynomial for $f$ about $b$:
\begin{equation}\label{e:taylorpolynomial}
T_n^bf(x):=\sum_{k=0}^n{f^{(n)}(b)(x-b)^k\over k!}.
\end{equation}
Typically, we treat this as a function in the variable $x$ and think of $b$ as being fixed.  However, one could also fix $x$ and consider $T_n^bf(x)$ as a function of $b$ instead.  For functor calculus, this point of view provides a way to think about Taylor towers for functors on categories of objects that naturally live under a fixed object as opposed to over a fixed object.   

In the setting of Goodwillie's functor calculus, expanding a Taylor tower about an object $B$ entails working with objects $X$ that all come equipped with a morphism to the fixed object $B$.  However, in some contexts, such as that of $k$-algebras for a fixed  ring (spectrum) $k$, it is more natural to consider objects $X$ that come with a morphism $k\rightarrow X$ {\it from} the fixed object $k$.  

In this paper, we work with Taylor towers from this perspective. For a functor $F:\AC\rightarrow \s$, from a category of objects under a fixed object $A$,   and an object $f:A\rightarrow X$ in $\AC$, we define $V_nF(f:A\to X)$ as the $n$th term in a {\it varying center} tower for $F$, that is, the analogue in the functor calculus context of the function obtained from $T_n^bf(x)$ by fixing $x$ and varying the center of expansion $b$.  The tower $\{ V_n F\}$ is strikingly different from the tower  $\{\Gamma_n^f F\}$ described earlier.  The $n$th term  is constructed from the degree $n$ functors $\Gamma^f_nF$  by using all choices of functors $\Gamma_n^f$ simultaneously (see Definition \ref{d:varying}). In brief, we define $V_nF(f:A\rightarrow X)$ by  
\[ V_nF(f:A\rightarrow X):= \Gamma_n^fF(A).\]  The analogous functions $T_n^bf(x)$, when considered as a function of $b$, need not be polynomial (see Remark \ref{r:notdegreen}), and we do not expect $V_nF$ to be a degree $n$ functor either.  However, we note that just as is the case for functions (see Remark \ref{r:converge}), the tower $\{V_nF\}$ exhibits peculiar convergence behavior: it approximates the constant functor with value $F(A)$ rather than the functor $F$ itself.

Goodwillie and Waldhausen proposed that  a tower like this varying center tower for the forgetful functor ${\mathcal U}$ from rational commutative ring spectra to $S$-modules, that is, from ring spectra under $H\mathbb Q$,  should play the role of the de Rham complex for rational commutative ring spectra.  We  refer to $V_{\infty} {\mathcal U}$ as the {\it Goodwillie-Waldhausen model} for the de Rham complex.  We explain why this is a reasonable model for the de Rham complex in the case of rational algebras  in Section 3.4 of this paper.  In particular,  for a fixed simplicial rational algebra $f:\mathbb Q_{\cdot}\rightarrow B_{\cdot}$ and object $\mathbb Q_{\cdot}\rightarrow X_{\cdot}\rightarrow B_{\cdot}$, that factors $f$, let $DR^n_{X_{\cdot}}B_{\cdot}$ denote the de Rham complex truncated at the $n$th stage:  
 \[
 DR^n_{X_{\cdot}}B_{\cdot}=\dots \leftarrow 0\leftarrow\dots\leftarrow 0\leftarrow \Omega^n_{X_{\cdot}}B_{\cdot}\leftarrow\dots\leftarrow \Omega^2_{X_{\cdot}}B_{\cdot}\leftarrow  \Omega^1_{X_{\cdot}}B_{\cdot}\leftarrow B_{\cdot}.
  \]
We show that as a functor of $X_{\cdot}$,  $DR^n_{X_{\cdot}}B_{\cdot}\simeq \Gamma_n^f{\mathcal U}(X_{\cdot})$ where $\mathcal U$ is the forgetful functor from algebras to modules.   This implies that the $n$-truncated de Rham complex  $ DR^n_{{\mathbb Q_{\cdot}}}B_{\cdot}$ is $V_n{\mathcal U}(B_{\cdot})$.

The model for the de Rham complex of a rational ring spectrum used by Rezk (in unpublished work \cite{Rezk}) is the cosimplicial ring spectrum $B\otimes _{H\mathbb Q} \sk _1{\Delta}_*^{\bullet}$.  Here, $\sk_1{\Delta}^\bullet_*$ is the $1$-skeleton of the standard simplices $\Delta^n$, assembled into a cosimplicial simplicial set.  The tensor product symbol $\otimes_{H\mathbb Q}$ denotes a coproduct, and is a generalization of the join construction (cf.  Theorem \ref{t:connected}).  For a finite set $U$, $B\otimes_{H\mathbb Q} U$ is the coproduct of $|U|$ copies of $B$ along the morphism ${H\mathbb Q}\to B$.  For rational algebras $A\to B$, Rezk analyzed the homotopy spectral sequence of the cosimplicial object $B\otimes _A \sk _1{\Delta}_*^{\bullet}$.  When the map $A\to B$ is a smooth morphism of $H\Q$-algebras, he determined that the $E_2$-page of this spectral sequence is essentially the de Rham complex and that there are no non-trivial differentials after this page.  Using this, he shows that  $\pi_*(B\otimes _{A} \sk _1{\Delta}_*^{\bullet})$ is isomorphic to the de Rham cohomology of the smooth map $A\to B$ of rational commutative algebras.  We refer to $B\otimes _{H\mathbb Q} \sk _1{\Delta}_*^{\bullet}$ as the {\it combinatorial model} for the de Rham complex.

The combinatorial  and  Goodwillie-Waldhausen models for the de Rham complex of a smooth map of rational
algebras both have natural extensions to suitable categories of $E_{\infty}$-algebras. Because crystalline cohomology is closely related to the de Rham complex, it is expected that a good generalization of the de Rham complex 
to $E_{\infty}$-algebras would also be useful for modelling crystalline cohomology analogues more generally. However,  two models for the de Rham complex of $E_{\infty}$-algebras need not be the same just because they agree rationally. Thus, the 
question this paper seeks to address is whether or not the  rational agreement of these two models for the de Rham complex remains  when the models are  extended to all $E_{\infty}$-algebras,  hence providing evidence that either of these may be regarded as candidates to model the de Rham complex of a map of
$E_{\infty}$-algebras. 
We prove that, not only rationally, but also for all $E_{\infty}$-algebras, the combinatorial and Goodwillie-Waldhausen models for the de Rham complex agree.  In fact, this equivalence is a special 
case of a much more general relationship that holds for a wide range of functors.  
This relationship is established in Theorem \ref{t:connected}.  In the statement of the theorem that follows, $|\cdot |$ denotes the geometric realization of a simplicial object, and $\Tot$ denotes the totalization of a cosimplicial object.   Recall that a homotopy functor is a functor that preserves weak equivalences.  
\vskip .1 in

\noindent {\bf Theorem \ref{t:connected}.} {\it Suppose that $F:\C \to \s$ is a homotopy functor where $\C$ is either $Top$ or $\s$, $f:A\rightarrow B$ is a $c$-connected  map in $\C$, and $F(X)\simeq \Gamma_\infty^fF(X)$ whenever $A\to X\to B$ is a factorization of $f$ with $X\to B$ $\rho$-connected.   
If $F$ commutes with the geometric realization functor, that is, the natural map $|F(X_{\cdot})|\rightarrow F(|X_{\cdot}|)$ is a weak equivalence for each simplicial object $X_{\cdot}$,  then
\[ V_\infty F(A\rightarrow B) \simeq \Tot |F(B\otimes_A sk_n\Delta^\bullet )|\]
whenever $n\geq \rho-c-1$.}
\vskip .1 in 
We note that the condition that $F$ commutes with the geometric realization functor is a mild hypothesis.  For example, \cite[Corollary 5.11]{AMO} shows that this is satisfied by any $n$-excisive functor to spectra which also commutes with filtered colimits of finite complexes.

 The forgetful functor ${\mathcal U}$ satisfies ${\mathcal U}(X)\simeq \Gamma^f_\infty\ {\mathcal U}(X)$ whenever $X\to B$ is  a 1-connected map. Thus, setting $F={\mathcal U}$, $A=H{\mathbb Q}$ and $n=1$ in this theorem yields the desired equivalence between the combinatorial model and the Goodwillie-Waldhausen model for the de Rham complex.  
 
 A key step in proving this theorem is the following proposition.  
 
 \vskip .1 in 
\noindent {\bf Proposition \ref{p:Tot}.}  {\it For each $n\geq 0$, the cosimplicial simplicial sets $\sk_n \Delta_*^\bullet$ are pro-equivalent to the empty set, thought of as a constant cosimplicial simplicial set with value $\emptyset$.  }
\vskip .1 in 
\noindent In other words, we view $\Tot^k\sk_n\Delta^{\bullet}_*$ as approximating the empty set, and by extension, 
$B\otimes _A\Tot^k\sk_n\Delta^{\bullet}$ as approximating $A$.  
 
 When combined with conditions that guarantee that the tower \[\{ |\Gamma_k^f F(X\otimes_A \sk_n \Delta^\bullet_*)|\}\] converges, Theorem \ref{t:connected}  is a direct consequence of Theorem \ref{t:infinity}, which we state below.  This  theorem relates the limit of the varying center  tower $\{V_nF(f:A\rightarrow B)\}$  to the totalization of the object obtained by evaluation of the limit of $\{\Gamma_n^fF\}$ on the  approximation of $A$ given by Proposition \ref{p:Tot}:
 \vskip .1in
 \noindent{\bf Theorem \ref{t:infinity}.} {\it Suppose that $F$ commutes with  realization, $f:A\rightarrow B$ is a morphism in $\C$, and $X$ is an object in $\Cf$.    For any $n\geq 0$, the tower of spectra
  \[
\{ V_kF(f:A\to B)\}_{k\geq 1} \]
is pro-equivalent to the tower of $\Tot$-towers of spectra
\[  \{\Tot^{m(k+1)} |\Gamma_k^fF(X\otimes_A \sk_n \Delta^\bullet_*)|\}_{k\geq 1}\]
\noindent{for} any integer $m\geq n+1$.}

\hskip -5pt An immediate consequence of this theorem is that \[V_\infty F(f:A\rightarrow B) \simeq \Tot |\Gamma_\infty^f F(X\otimes_A \sk_n \Delta_*^\bullet)|,\] since pro-equivalences of towers imply that their homotopy inverse limits are weakly equivalent.
 
This theorem can be compared to Theorem 4.0.2 of the second author's thesis \cite{Rosona} (or the published work, \cite{Rosona-AGT}), which says that there are weak equivalences $P_nF(A)\to \Tot^{(k+1)n}P_nF(X\otimes_A \sk_n \Delta_*^\bullet)$ for each $k$ and $n$, where $F$ is a homotopy functor which commutes with colimits, and $P_nF$ is the $n$th term in Goodwillie's Taylor tower for $F$.    
There is an equivalence $P_nF(A)\simeq \Gamma_n^fF(A)$ where $f:A\to X$, see \cite[6.8, 6.9]{BJM}.  Together with the fact that the tower $\{ V_nF\}$ is constructed using all possible choices of the functors $\Gamma_n^fF$ simultaneously, it is not surprising that there is a similarity between \cite[Theorem 4.0.2]{Rosona} and Theorem \ref{t:infinity}.  However, just as $\Gamma_n^fF$ is only a part of the construction of the tower $\{V_nF\}$, the finite stage $\Tot^{(k+1)n}P_nF(X\otimes_A \sk_n \Delta_*^\bullet)$ is only part of the $\Tot$-tower considered in Theorem \ref{t:infinity}.  

 We also extend Theorem \ref{t:connected}  to  functors from arbitrary simplicial model categories, $\C$,  to spectra.   
 
 \vskip .1 in 
\noindent {\bf Corollary \ref{p:analyticprep}.}   {\it Let $\C$ be a simplicial model category and let $F:\C\rightarrow \s$ be a  homotopy functor that commutes with the geometric realization functor.  Let $f:A\rightarrow B$ be a morphism in $\C$ and $X$ be an object in $\C$ with morphisms $A\rightarrow X\rightarrow B$ that factor $f$.   Suppose there exists an $n\geq 0$ for which we have an equivalence
\[  |F(X\otimes_A sk_n \Delta^k)| \simeq  |\Gamma_\infty^fF(X\otimes_A sk_n\Delta^k)| \]
for all cosimplicial degrees $k\geq 0$.  Then $$\Tot |F(X\otimes_A sk_n \Delta^\bullet)| \simeq V_\infty F(f:A\rightarrow B).$$
}
 \vskip .1 in

The paper is organized as follows.  In Section 2, we review basic results about simplicial and cosimplicial objects, and prove Proposition \ref{p:Tot}.  In addition, we review some ideas concerning pro-equivalences of towers associated to cosimplicial objects.   We review the construction and basic properties of the Taylor tower of \cite{BJM} in Section 3.  We then use this Taylor tower  to define the varying center towers $\{V_nF\}$ described earlier in this introduction.  We finish Section 3 by explaining why the de Rham complex can be treated as a varying center tower for the functor $\mathcal U$.  
We devote the final section of the paper to proofs of Theorems \ref{t:infinity} and \ref{t:connected}, and Corollary \ref{p:analyticprep}.  
 \vskip .1 in 
 \noindent {\bf  Acknowledgments:}   We are indebted to Philip S. Hirschhorn for providing us with Lemma \ref{p:Phil} (found in \cite{Hirschhorn}) which allowed us to complete the proof of  Proposition \ref{p:fix}. The authors were able to meet and work together several 
times because of the generosity and hospitality of the following:  the Midwest Topology Network (funded by NSF grant DMS-0844249), the Union College Faculty Research Fund, the Pacific Institute for the Mathematical Sciences, and the mathematics department of the  University of Illinois at Urbana-Champaign.   We thank them for their support.     We would like to thank Charles Rezk for several helpful conversations.  We also 
thank Tom Goodwillie for his ideas that contributed to the development of Section 3.

%%%%%%%%%%%%%%%%%%%%%%%%%%%%%%%%%%%%%%%%%%%%

\section{Cosimplicial objects and $\Tot$}

%%%%%%%%%%%%%%%%%%%%%%%%%%%%%%%%%%%%%%%%%%%%%%

In this paper,  we  provide a combinatorial model for the limit of a certain tower of degree $n$ approximations to a homotopy functor.   In this section, we provide relevant background on simplicial and cosimplicial objects.  We demonstrate how the empty set can be obtained from $n$-connected spaces in a systematic way in Proposition \ref{p:Tot}.  This seemingly simple fact about cosimplicial spaces leads to the equivalence of Theorem \ref{t:connected}.  We review towers of objects associated to cosimplicial objects, and introduce pro-equivalences between these as our preferred notion of ``equivalence" for towers.  We finish the section by recording a useful lemma concerning multicosimplicial objects.  For an introduction to simplicial and cosimplicial objects, the reader is referred to  \cite{BK} and \cite{GJ}. 

Let $\bf \Delta$ be the category of non-empty finite totally ordered sets and order-preserving functions.    A cosimplicial object in any category $\C$ is a covariant functor $X^{\bullet}:{\bf \Delta} \to \C$.  Let $X^m$  denote the value of $X^\bullet$ at $[m]=\{ 0, 1, \cdots , m\}$. 
 Morphisms of cosimplicial objects in $\C$ are natural transformations of functors.   The category of cosimplicial objects and natural transformations between them is denoted ${\C}^{\bf \Delta}$.
Dually, a simplicial object is a contravariant functor $Y_{*}:{\bf \Delta} \to \C$, and its $k$th object is $Y_k$.  
In this section the category $\C$ will be one of simplicial sets,  
spaces, or spectra, ${\s}$, depending on the context.  The model category of (unpointed) simplicial sets will be the usual one, as in \cite[Definition 7.10.7]{H}.  We require our category of (unpointed) topological spaces to satisfy the criteria set out in \cite[7.10.2]{H}, and we use the model category structure explained in \cite[Definition 7.10.6]{H}.  By spectra, we mean a suitable model category of spectra with symmetric monoidal smash product, such as the ones defined in \cite{HSS} or \cite{EKMM}.

We make use of the standard cosimplicial simplicial set 
$\Delta_{*}^{\bullet}$ given by:
\[ \Delta^n_k = {\rm hom}_{\bf \Delta} ([k], [n]).\]  By taking the geometric realization of this cosimplicial simplicial set in each cosimplicial degree we obtain a cosimplicial space which we denote $\Delta^{\bullet}$.  

\begin{defn}  The  totalization of any cosimplicial object $X^\bullet$ is
\[ \Tot X^\bullet := {\operatorname{holim}}_{\bf \Delta} X^{\bullet}.\]
\end{defn}

We take this as the definition of the totalization because it is  homotopy invariant.  We will assume that in this paper we have selected functorial homotopy limits everywhere satisfying the conditions enumerated in \cite[Lemma 2.5]{BJM}.  In many situations, it is convenient to use a {\it particular} model for the homotopy limit defining $\Tot$ (see, e.g. \cite{BK}).  This leads to an alternate description of the totalization as
\[ \tot X^\bullet := {\rm Hom}_{\C^\Delta}(\Delta^\bullet, X^\bullet),\]
where $\rm Hom$ denotes the appropriate function complex. 
In particular, if $X$ is a fibrant object in $\C^\Delta$ (with the Reedy model structure, see \cite{H}), there is a weak equivalence $\Tot X^\bullet \simeq\tot X^\bullet$ (as in \cite[Lemma 2.12]{GJ}).  The latter part of this equivalence is often taken as the definition of the totalization.

We can filter $\bf \Delta$ by full subcategories ${\bf \Delta}_r$ consisting of sets with at most $r+1$ elements.  The inclusion  $i_r:{\bf \Delta}_r\to {\bf \Delta}$ induces a truncation functor  $\C^{\bf \Delta} \to \C^{{\bf \Delta}_r}$ given by sending $X^\bullet$ to its restriction $X^\bullet\circ i_r$.  The totalization of $X$ inherits a filtration from $\bf \Delta$.

\begin{defn} 

The $r$th totalization of a cosimplicial object $X^\bullet$ is
\[\Tot^r X^\bullet := \holim_{{\bf \Delta}_r} (X^\bullet\circ i_r) . \]  
\end{defn}

The functors $\Tot^r$ assemble into a tower $\Tot^r X^\bullet \to \Tot^{r-1} X^\bullet$ for $r\geq1$.  
 We denote this tower by
\[ \tower X^\bullet := \{ \Tot^r X^\bullet \}_{r\geq 0}.\]  We set
\[ \tot^r X^\bullet := \operatorname{hom}_{\C^{\bf \Delta}} ( \sk_r\Delta^\bullet,  X^\bullet),\]
where  $\sk_n Y_*$ denotes the $n$-skeleton of a simplicial space, as in  section IV.3.2 of \cite{GJ}.
This is a model for $\Tot^r X^\bullet$ in the same way that $\operatorname{tot} X^\bullet$ is a model for $\Tot X^\bullet$.  In particular, if $X^\bullet$ is  Reedy fibrant we have $\operatorname{tot}^r X^\bullet \simeq \Tot^r X^\bullet$.

To compare two cosimplicial objects $X^\bullet$ and $Y^\bullet$, we could ask when a map $f:X^\bullet \to Y^\bullet$ yields an {\it equivalence}.  In the Reedy model structure on cosimplicial spaces or spectra, weak equivalences are defined to be levelwise equivalences $X^n \to Y^n$ for each $n\geq 0$.   This notion is too strong for our purposes.  We could instead ask that the associated map  $\Tot X^\bullet \to \Tot Y^\bullet$ is a weak equivalence in the category $\C$.  However, this notion misses the topology coming from the inverse limit tower.  To incorporate the structure of the $\Tot$-tower in equivalences of  cosimplicial objects, we want equivalences which are weaker than Reedy equivalences but stronger than  equivalences in $\C$.  
\begin{defn} \label{d:pro} \cite[Chapter III]{BK}  A map of towers $\{f_s\}: \{ A_s\} \to \{B_s\}$ in any category is a pro-isomorphism if for each $s$ there is a $t$ and a map $B_{s+t}\to A_s$ making the following diagram commute:
\[ \xymatrix{ A_{s+t}\ar[r]^{f_{s+t}} \ar[d] & B_{s+t} \ar[d] \ar[dl] \\ A_s \ar[r]_{f_s} & B_s}.\]  For pointed towers of connected spaces or spectra, a map of towers is a weak pro-homotopy equivalence if it induces a pro-isomorphism of sets on $\pi_0$ and a pro-isomorphism of groups on $\pi_n$ for each $n\geq q$.  
\end{defn}
We will shorten ``weak pro-homotopy equivalence" to ``pro-equivalence".  A pro-equivalence of $\Tot$-towers $\tower X^\bullet \to \tower Y^\bullet$ automatically induces a weak equivalence  of objects $\Tot X^\bullet \simeq \Tot Y^\bullet$ (see \cite[Chapter III Prop. 2.6]{BK}), but need not have $\Tot^r X^\bullet \simeq \Tot^r Y^\bullet$ for any $r$.  
On the other hand, if there is a weak equivalence $\Tot^r X^\bullet \simeq \Tot^r Y^\bullet$ for all $r\geq n$, then the towers $\tower X^\bullet$ and $\tower Y^\bullet$ are automatically pro-equivalent.

The next few results show that the cosimplicial spaces $\sk_n \Delta^\bullet$ approximate the empty set by $n$-connected spaces in a systematic way.  

\begin{lem}\label{l:empty} Let $n\geq 0$.  The space $\operatorname{tot}^{k} \sk_n\Delta^\bullet$ satisfies \[\operatorname{tot}^k \sk_n \Delta^\bullet=\emptyset \] for all $k\geq n+1$. \end{lem}

\begin{proof} An element of ${\rm tot}^ksk_n\Delta^{\bullet}$ can be treated as a sequence of maps 
\[\alpha _i: \sk_k\Delta ^i\rightarrow {\operatorname{sk}}_n\Delta ^i
\] that commutes with the cosimplicial face and degeneracy maps.   We claim that no such sequence can exist when $k\geq n+1$.   To see why, suppose that we have a commuting diagram   
\begin{equation} \label{e:emptyproof} \xymatrix{ \Delta^{n} \ar[r]^{d^i} \ar[d]^{\alpha_{n}} & \Delta^{n+1} \ar[d]^{\alpha_{n+1}} \\ \sk_n \Delta^{n} \ar[r]^{d^i} & \sk_n \Delta^{n+1} }\end{equation}
where the horizontal maps $d^i$ are any of the $n+2$ cosimplicial face maps.   Since the $\alpha_{i}$'s are required to commute with the cosimplicial structure maps,  $\alpha_{n+1}$ restricted to the boundary of $\Delta^{n+1}$ is required to be homotopic to the identity map.  However, this is impossible since $\alpha_{n+1}$ extends to $\Delta^{n+1} \cong D^{n+1}$, meaning that $\alpha_{n+1}$ is contractible.  Since there are no contractible maps whose image is all of the $n$-sphere $sk_n\Delta^{n+1}$, this shows that \\ $\operatorname{hom}_{\C^{\bf \Delta}}(\Delta^\bullet, \sk_n \Delta^\bullet)$ is empty.  

Since $\sk_k \Delta^n = \Delta^n$ whenever $k\geq n$, the same argument also shows that $\operatorname{hom}_{\C^{\bf \Delta}}(\sk_k\Delta^\bullet, \sk_n \Delta^\bullet)$ is empty whenever $k\geq n+1$.

\end{proof}

\begin{prop}\label{p:Tot}  For each $n\geq 0$, the cosimplicial simplicial sets $\sk_n \Delta_*^\bullet$ are pro-equivalent to the empty set, thought of as a constant cosimplicial simplicial set with value $\emptyset$.  Moreover, we have
\[ \Tot^k \sk_n \Delta_*^\bullet = \emptyset \]
whenever $k\geq n+1$.
\end{prop}

\begin{rem} Note that $\sk_n \Delta_*^\bullet$ is not a fibrant cosimplicial simplicial set.  Therefore Lemma \ref{l:empty} and Proposition \ref{p:Tot} are distinct statements. \end{rem}

\begin{proof}[Proof of Proposition \ref{p:Tot}]  
Suppose that $X_*$ is a  simplicial set.  To show that $X_*= \emptyset$, it suffices to show that the realization of the simplicial set  $|X_*|= \emptyset$.  This is sufficient because the realization of a non-empty simplicial set must be non-empty.  We apply this to the simplicial set $X_*=\Tot^k \sk_n \Delta_*^\bullet$.   
The geometric realization commutes with finite limits (\cite[Chapter I, Prop. 2.4]{GJ} or \cite[Prop. 3.2.3]{Hovey}), and by using Reedy fibrant replacement one sees that it also commutes with finite homotopy limits such as $\Tot^k$.
 
  We obtain
\[|\Tot^k\sk_n\Delta^{\bullet}_*|\simeq \Tot^k|\sk_n\Delta^{\bullet}_*|.
\] 
%All topological spaces are fibrant, so in particular $|\sk_n \Delta^i_*|$ is fibrant for each $i$.
By \cite[Theorem 19.8.4(2)]{H} we have a natural weak equivalence
\[\Tot^k|\sk_n\Delta^{\bullet}_*|\simeq \tot^k|\sk_n\Delta^{\bullet}_*|=\tot^k \sk_n\Delta^\bullet.
\]
By Lemma \ref{l:empty}, the right hand side is equal to $\emptyset$.
The conclusion follows by noting that the only space weakly equivalent to the empty set is the empty set itself.
\end{proof}

We end this section with a kind of cosimplicial Eilenberg-Zilber Theorem.  
A $k$-multicosimplicial object in $\C$ is a functor $X$ from 
${\bf \Delta}^{\times k}$, the $k$-fold product of the category $\bf \Delta$ with itself, to $\C$.   
The diagonal cosimplicial object, $\diag X^\bullet$, is obtained from $X$ by precomposing with the diagonal functor $ {\bf\Delta} \to {\bf\Delta}^ {\times k}.$  Let $\Tot^j_i$  
denote the homotopy limit $\holim_{{\bf \Delta}_j} X$ taken in the $i$-th variable only.  

\begin{lem}\label{p:Phil}\cite{Hirschhorn}  The diagonal functor preserves Reedy fibrations.  In particular, if $X$ is a Reedy fibrant multicosimplicial object of $\C$, then $\diag X$ is a Reedy fibrant cosimplicial object of $\C$.
\end{lem}

This lemma implies that the diagonal is well-behaved with respect to taking homotopy limits.  In particular, it implies that if $X\simeq Y$ is a Reedy weak equivalence of Reedy fibrant multicosimplicial objects, then 
\[ \holim_I \diag X \simeq \holim_I \diag Y\]
for any indexing category $I$.

The following lemma appears as \cite[Lemma 3.3]{Rosona}, where it is proved in the case where $\C$ is the category of simplicial sets.  The proof presented there works for any category $\C$ enriched over simplicial sets.

\begin{lem}\label{l:Rosona}\cite[Lemma 3.3]{Rosona} Let $X$ be a $k$-multicosimplicial object in $\C$.  Then 
\[ \Tot_1^{n_1} \cdots \Tot_k^{n_k} X \simeq \Tot^{N} \diag (\csk_1^{n_1}\cdots \csk_k^{n_k} X)\]
where $N= n_1 + \cdots + n_k$ and $\csk_i^{n_i} X$ is the $n_i$-th coskeleton of $X$ in the $i$-th variable for all $1\leq i \leq k$.
\end{lem}

These two lemmas immediately imply the following proposition, which is a cosimplicial version of the Eilenberg-Zilber Theorem.

\begin{prop} \label{p:fix}  Let $X$ be a Reedy fibrant $k$-multicosimplicial object of $\C$ such that
\[ X\simeq \csk_1^{n_1}\cdots \csk_k^{n_k} X\]
is a Reedy weak equivalence.  Then 
\[ \Tot_1^{n_1} \cdots \Tot_k^{n_k} X \simeq \Tot^{N} \diag X\]
where $N=n_1+\cdots + n_k$.
\end{prop}

%%%%%%%%%%%%%%%%%%%%%%%%%%%%%%%%%%%%%%%%%%%%%%%%%%

\section{Calculus for functors of $\AC$}

%%%%%%%%%%%%%%%%%%%%%%%%%%%%%%%%%%%%%%%%%%%%%%%%%%

The main results of this paper, which we prove in Section 4, are stated in terms of  the varying center tower for functors whose source categories are  categories whose objects live naturally {\it under} a fixed initial object.  The primary goal of this section is to define these towers and explain how they are related to Taylor towers, i.e., the towers of functors described in the second paragraph of the introduction.   Existing models for  Taylor towers use  categories whose objects live naturally {\it over} a fixed terminal object.  The definition of the varying center tower will use the discrete Taylor tower $\{ \Gamma_n^f F\}$ of \cite{BJM}, which is defined for functors whose source categories consist of objects that factor a fixed morphism, $f$.  To construct the varying center tower, we combine the discrete Taylor towers corresponding to all possible values of $f$.  In Section 3.1 we  review  these discrete Taylor towers.  

Let $\AC$ be the category of objects in a simplicial model category $\C$ under a fixed object $A$, with morphisms given by commuting triangles.  As part of the process of extending the construction of Section 3.1 to the category $\AC$, in Section 3.2 we explain how the new varying center tower serves as an analogue of the series obtained by fixing the variable of a Taylor series and allowing the center of expansion to vary.  The construction of the varying center tower is carried out in Section 3.3.   We conclude the section by describing how  the de Rham complex can be viewed as an example of one of these varying center towers.

Throughout this section, we let $\C$ be a simplicial model category, $\s$ be a category of spectra, and  
$\star$ be the initial/final object in $\s$.   We assume that all functors are homotopy functors; that is, that they preserve weak equivalences.   

\subsection {Taylor towers for $\Cf$}
Let $f:A\to B$ be a morphism in $\C$ and let $\Cf$ be the category whose objects factor the map $f$.  The objects of $\Cf$ are triples $(X, \alpha_X, \beta_X)$ where $X$ is an object of $\C$ and $\alpha_X: A\to X$ and $\beta_X:X\to B$ are morphisms such that $\beta_X\circ \alpha_X= f$.  When the maps $\alpha_X$ and $\beta_X$ are understood, we  use $X$ to denote the triple $(X, \alpha_X, \beta_X)$.  Morphisms are given by the obvious commuting diagrams.  In order for our constructions to be homotopy invariant, we assume that all objects of $\Cf$ are cofibrant, i.e., that $f:A\to B$ is a cofibration in $\C$ and each $\alpha_X$ is a cofibration.

To review  existing Taylor towers for functors from $\Cf$, we begin by summarizing the two notions of polynomial degree $n$ functors considered in \cite{BJM} and \cite{G2}, respectively.    For full details, see  Sections 3.4 and 4 of \cite{BJM} and Section 3 of \cite{G2}.   The 
first notion depends on cross effects functors.  Let $\coprod_A$ denote the coproduct in $\Cf$.  Recall that an $n$-cubical diagram in a category $\D$ is a functor from the power set $P({\bf n})$ of ${\bf n}=\{ 1, \ldots , n\}$ to $\D$.

\begin{defn}\label{d:xfx} \cite[Definition 3.4]{BJM} Let ${\bf X}$ be an $n$-tuple $(X_1, \ldots , X_n)$ of objects in $\Cf$.  
\begin{itemize}
\item Let $(\coprod_n)_B^{\bf X}:P({\bf n})\to \s$ be the $n$-cubical diagram defined on $S\subseteq {\bf n}$ by  
\[ \label{e:coprod_n} (\coprod_n)_B^{{\bf X}}(S) = X^1_B(S)\coprod_A \cdots \coprod_A X^n_B(S)\]
where $X^i_B(S) = X_i$ when $i\notin S$ and $X^i_B(S)=B$ otherwise.  If $S\subset T$, the morphism $(\coprod_n)_B^{{\bf X}}(S\subset T)$
is induced by the maps $\beta_{X_i}$  on summands $X^i_B(S)$ with  $i\in T-S$ and identity maps on summands with $i\in S$.

\item If $F:\Cf \to \s$, the $n$th cross effect of $F$ relative to $f$ evaluated at ${\bf X}$, $cr_n^fF(X_1,\dots, X_n)$,  is  the iterated homotopy fiber (see \cite[Definition 3.2]{BJM}) of the cubical diagram obtained by applying $F$  to $(\coprod_n)_B^{\bf X}.$ 
\end{itemize}
Since  the $n$-cubical diagram  $(\coprod_n)_B^{{\bf X}}$ and the homotopy fiber are functorial in ${\bf X}$, this defines a functor $cr_n^fF:{\sc C}_f^{\times n} \rightarrow \s$, called the $n$th cross effect functor.
\end{defn}

When $n=2$, the 2-cubical diagram $(\coprod_2)_B^{{\bf X}}$ of Definition \ref{d:xfx} is simply the square diagram
\[ \xymatrix{ X_1 \coprod_A X_2 \ar[r]^{1\coprod\beta_{X_2}} \ar[d]_{\beta_{X_1}\coprod 1} &X_1\coprod_A B \ar[d]^{\beta_{X_1}\coprod 1}\\ 
B\coprod_A X_2 \ar[r]_{1\coprod \beta_{X_2}} &B\coprod_A B}\]
and $cr_2F(X_1, X_2)$ is the iterated homotopy fiber of the diagram obtained by applying  $F$ to this square.

\begin{defn} \label{d:rel-f} \cite[Definition 3.21]{BJM} A functor $F:{\Cf} \rightarrow \s$ is {\it degree $n$ relative to $f:A\rightarrow B$} provided that $cr_{n+1}^fF\simeq \star$.   
We note that this condition was simply called ``degree $n$" in \cite{BJM}.   
\end{defn}

This is a weaker condition than the $n$-excisive condition satisfied by the $n$th term 
in Goodwillie's Taylor tower, defined as follows.

\begin{defn}  \label{d:n-excisive} \cite[Definition 3.1]{G2}
A functor $F:{\sc D}\rightarrow \s$, where ${\sc D}$ is $\C, \AC,$ or $\Cf$, is $n$-excisive provided that when $F$ is applied to any strongly cocartesian $(n+1)$-cubical diagram, $\chi: P({\bf n})\to \sc D$, the result is a homotopy cartesian diagram in $\s$.    
\end{defn} 
The properties of degree $n$ relative to $f$ and $n$-excisive are related by the following 
definition and proposition.  

\begin{defn}   \cite[Definition 4.2]{BJM}
A functor $F:{\sc D}\rightarrow \s$, where $\sc D$ is $\C$, $\AC$ or $\Cf$,  is $n$-excisive {\it relative to $A$} provided that  when $F$ is applied to any strongly cocartesian $(n+1)$-cubical diagram $\chi: P({\bf n})\to \sc D$ {\it with $\chi(\emptyset)=A$},  the result is a homotopy cartesian diagram in $\s$.    \end{defn}

\begin{prop} \cite[Propositions 4.3, 4.11]{BJM} \label{p:deg-exc} Let $f:A\to B$, and let $F:\Cf\to \s$.  \begin{enumerate}
 \item The functor  $F$ is 
degree $n$ relative to $f$ if and only if $F$ is $n$-excisive relative to $A$.   
\item   If $F$ commutes with 
realizations, that is, if the natural map $|F(X_{\cdot})|\rightarrow F(|X_{\cdot}|)$ is an equivalence for all simplicial objects $X_{\cdot}$, then $F$ is $n$-excisive relative to $A$ if and only if $F$ is $n$-excisive.   
\end{enumerate}
\end{prop}

The functor $\perp_nF$ obtained by precomposing $cr_{n+1}^fF$ with the diagonal, $\perp_n^fF(X) := cr_n^fF(X, \ldots , X)$, forms a cotriple (or comonad) on the  category of homotopy functors from $\Cf$ to $\s$ \cite[Theorem 3.17]{BJM}. Standard results about cotriples yield a simplicial spectrum $(\perp_{n+1}^f)^{*+1}F(X)$ whose $k$th spectrum is given by iterating the diagonal cross effects construction $k+1$ times.  One can then define
\[ \Gamma_n^fF(X) \, := \, \operatorname{hocof} \left( |(\perp_{n+1}^f)^{*+1}F(X)| \to F(X) \right) \]
as in \cite[Definition 5.3]{BJM}.  These functors assemble into the tower in Diagram (\ref{e:taylortower}) of functors and natural transformations \cite[Theorem 5.8]{BJM} and satisfy:
%\begin{itemize}
%{\color{red} Theorems 5.4, 5.6 and 5.8 }of \cite{BJM}  indicate that the functors $\Gamma_n^f F$ assemble into a tower
%The following theorem constructs the expected Taylor tower of functors which are degree $n$ relative to $f$.
%
%\begin{thm}\label{t:tower} (\cite {BJM}, 5.4, 5.6, 5.8)  Let $F:\Cf \rightarrow \s$.   
%There is a tower 
%of functors and natural transformations:
%\[ \xymatrix{ && F \ar[dl]_{\gamma_{n+1}^f}\ar[d]^{\gamma_n^f}\ar[dr]^{\gamma_{n-1}^f}\quad \\
%\cdots \ar[r] & \Gamma_{n+1}^fF \ar[r]^{q_{n+1}^f} & \Gamma_n^fF \ar[r]^{q_n^f\quad} & \Gamma_{n-1}^fF \ar[r] ^ \cdots \ar[r] & \Gamma_1^fF \ar[r] & \Gamma_0^fF}\]
%where each
\begin{enumerate}
\item $\Gamma _n^fF$ is degree $n$ relative to $f$ \cite[Proposition 5.4]{BJM},  
\item  $\Gamma_n^fF$ is universal, up to weak equivalence, among functors that are  degree $n$ relative to $f$ and have natural transformations from $F$ \cite[Proposition 5.6]{BJM}.   
\end{enumerate}

%The homotopy inverse limit of this tower is denoted $\Gamma _{\infty}^fF$.
%\end{thm}

 We now turn our attention to Goodwillie's construction.  In \cite[Section 6.1]{BJM} we extended  Goodwillie's construction 
of Taylor towers for functors of spaces or spectra \cite[Section 1]{G3}  to functors $F:\Cf \rightarrow 
\s$ as follows.   
For any finite set $U$  of cardinality $u$ and object $X$ in $\Cf$, let $B\otimes _X U$ be the homotopy colimit over $u$ copies of  the map $\beta_X:X\to B$ out of a single domain $X$.   This is a generalization of the fiberwise join construction found in \cite[Section 1]{G3},  denoted there as $X*_B U$.   The $n$th term in Goodwillie's Taylor tower is given by
\begin{equation} \label{e:tom's} P_nF(X) := \operatorname{hocolim}_k T^k_nF(X) \end{equation}
where $T_nF(X) := \operatorname{holim}_{U\in P_0({\bf n+1})} F(B\otimes _XU)$ and the homotopy inverse limit is taken over $P_0({\bf n+1})$, the partially ordered set of non-empty subsets of $\{ 1, \ldots , n+1\}$.  %For further details, including a construction of the maps involved in this homotopy limit construction, see Section 6 of \cite{BJM}.

When evaluated at the initial object, $A$,  of $\Cf$, we show that $\Gamma_n^fF(A)$ agrees with Goodwillie's $n$-excisive polynomial functor $P_nF(A)$.
\begin{thm}   \cite[Corollary 6.8, Theorem 6.9]{BJM}  Let $F:\Cf \rightarrow \s$.   There is a weak homotopy equivalence of spectra $\Gamma _n^fF(A)\simeq 
P_nF(A)$.   Moreover, if $F$ commutes with realization, then for any object $X$ in $\Cf$, there is a weak homotopy equivalence of spectra
$\Gamma _n^fF(X)\simeq P_nF(X)$.   
\end{thm}  
This weak homotopy equivalence allows us to apply  results of  \cite{G3} to the functor $\Gamma_nF$; for example we have the following proposition which provides a classification of the homogeneous layers.    
  
  \begin{prop}\label{p:layers} %\cite[Help?]{G3} 
  Let $F:\Cf \rightarrow \s$ and $X$ be an object in $\Cf$.  Suppose that $F$ commutes with the geometric realization functor, or that $X=A$.  Let $\widetilde {\Gamma} _n^fF(X)$ denote  the homotopy fiber of the natural map $q_n^f: \Gamma_n^fF(X)\to \Gamma_{n-1}^fF(X).$ Then there exists a functor of $n$ variables, $\widetilde {L}_n^fF: {\Cf}^{\times n} \to \s$ such that 
  \begin{enumerate}\item  $\widetilde {L}_n^fF$ is degree $1$ relative to $f$ in each variable,
  \item $\widetilde {L}_n^fF$ is $n$-reduced, that is, it is equivalent to $\star$ when evaluated at $B$ in any of its variables, \item  for any permutation $\sigma\in \Sigma _n$, there is a  natural weak equivalence \[\widetilde {L}_n^fF(X_{\sigma(1)}, \dots, X_{\sigma(n)})\simeq \widetilde {L}_n^fF(X_1,\dots, X_n),\]  and 
 \item  when evaluated at $X$, $\widetilde {\Gamma} _n^f F(X)$ is equivalent to 
  \[ [\widetilde {L}_n^fF(X, \dots, X)]_{h\Sigma_n}\]
   where $h\Sigma _n$ denotes the homotopy orbits with respect to the natural $\Sigma _n$ 
  action that permutes the variables of $\widetilde {L}_n^fF$.  
  \end{enumerate}
  
  \end{prop}
  
\begin{proof} Goodwillie's proof of Theorem 4.1 and Corollary 4.2 in \cite{G3} can be applied to the constructions $Y\otimes _X U$, $T_nF$ and $P_nF$ as we have defined them above.   This proves the stated result since $P_nF(X) \simeq \Gamma_n^fF(X)$ under either hypothesis.
\end{proof}

 %%%%%%%%%%%%%%%%%%%%%%%%%%%%%%%%%%%%%%%%%%%%%%%%%%%%%%
\subsection {Motivation from the Taylor series of  a function}
 %%%%%%%%%%%%%%%%%%%%%%%%%%%%%%%%%%%%%%%%%%%%%%%%%%%%%%
  
  %The category $\Cf$ is  more rigid than what we need:  it has both an initial object and
% a final object, and every object factors a particular map $f$.  It is natural to view the category of  $k$-algebras, the setting for the classical algebraic de Rham complex (see, e.g., Chapter 9 of \cite{weibel}), for a fixed commutative ring $k$  as a category of objects under $k$, but not as a category over a fixed object.   We would like to further generalize the idea of degree $n$ approximations to settings like these,  %However, we still need an initial object  to ``add" over in order to form the coproducts that are the basic building blocks of the degree $n$ polynomial approximations.
%
%so we will consider functors from the category $\AC$. 

 For a morphism $f:A\rightarrow B$, we have defined $\Gamma_n^f F$ to be the  degree $n$ approximation to $F$ relative to $f$.   We consider this to be analogous to a degree $n$ polynomial approximation $T_n^b{\mathfrak f}(x)$ of Equation (\ref{e:taylorpolynomial}) where $A\to X$ plays the role of $x$ and $B$ plays the role of $b$.   The object $B$ should be viewed as a ``center of expansion" for $\Gamma_n^fF$ in the following sense.   In contexts where we can measure the connectivity of a map $X\rightarrow B$ (such as spaces or spectra), for a suitably nice functor $F$, the functor $\Gamma_n^fF$ approximates $F$ in the sense that if the map $X\rightarrow B$ is $k$-connected for $k$ above some constant $\kappa$ that depends on $F$, then the map $F(X)\rightarrow \Gamma_n^fF(X)$ is on the order of $(n+1)k$-connected.  That is, for objects that are within some distance of  $B$ homotopically, $\Gamma_n^fF(X)$ is an approximation to $F(X)$ that improves as $n$ increases.  See, for example, \cite[Proposition 1.6]{G3}.

%For functors of spaces or spectra, the terminal object is the trivial object $\star$, and the connectivity of the map $X\to \star$ measures the connectivity of $X$.  Since convergence of Goodwillie's tower $P_nF$ for analytic functors depends on the connectivity of $X$ (see \cite[Theorem 1.13]{G3}), the terminal object is analogous to the center of expansion of a Taylor series for a function.  To shift our focus to functors from categories which do not have a fixed terminal object, we first considered the analogous situation for a Taylor series of a function in Equation (\ref{e:taylorpolynomial}) of the introduction.

%Let ${\mathfrak f}\in C^\infty({\mathbb R})$.  For any point $b\in {\mathbb R}$, the $n$th Taylor polynomial of $\mathfrak f$ centered at $b$ is 
%\[ T_n^b{\mathfrak f}(x) = \sum_{k=0}^n \frac{{\mathfrak f}^{(k)}(b)(x-b)^k}{k!} .\]
%This is the degree $n$ polynomial approximation of the function ${\mathfrak f}(x)$ near the point $b$.  We expect that for $x$ near $b$,  the Taylor polynomials $T_n^b{\mathfrak f}$ become better and better approximations of the function $\mathfrak f$ as $n$ increases.  

We make three observations about the Taylor series of a function ${\mathfrak f}$ that motivate our
study of polynomial approximations for functors from categories of objects under a fixed initial object.

\begin{Observation}\label{r:Tnb}   In considering $T_n^b{\mathfrak f}$, we can  change our perspective from $x$ to $b$.  That is, we can view $T_n^b{\mathfrak f}(x)$  as a function of $b$ by evaluating at a fixed $x$, say $x=0$ for simplicity:
\[ T_n{\mathfrak f}(b) = T_n^b{\mathfrak f}(0) = \sum_{k=0}^n \frac{{\mathfrak f}^{(k)}(b)(-1)^k(b)^k}{k!} .\]
In Definition \ref{d:varying}, we explain how to implement this observation for functors.
\end{Observation}
\begin{Observation} \label{r:converge}  When we shift our focus from $x$ to $b$, we observe that 
$T_n{\mathfrak f}(b)$ is not trying to approximate the {\it function} $\mathfrak f$, but only the discrete value ${\mathfrak f}(0)$. In particular, $T_n{\mathfrak f}(b)=f(0)$ for $b\neq 0$ if and only if ${\mathfrak f}^{(n+1)}(b)=0$ for all $b$.  For, if 
\[ {\mathfrak f}(0)=  \sum_{k=0}^n \frac{{\mathfrak f}^{(k)}(b)(-1)^k(b)^k}{k!} \]
then differentiating both sides with respect to $b$ yields
\[ 0 = \sum_{k=0}^n \left( \frac{{\mathfrak f}^{(k+1)}(b)(-1)^kb^k}{k!} + \frac{{\mathfrak f}^{(k)}(b)(-1)^kb^{k-1}}{(k-1)!}\right). \]
This is a telescoping sum, and hence we obtain $\frac{{\mathfrak f}^{(n+1)}(b)(-1)^nb^n}{n!}=0$.  If $b$ is non-zero, then this equation holds if  $f^{(n+1)}(b)=0$.  Thus, $T_n{\mathfrak f}(b) = {\mathfrak f}(0)$ if and only if ${\mathfrak f}$ is a polynomial of degree $n$.  Compare this to Proposition \ref{p:constantconverge}.
\end{Observation}
\begin{Observation} \label{r:notdegreen}  When considered as a function of $b$ with $x=0$ fixed,
$T_n{\mathfrak f}(b)$ is not necessarily a degree $n$ polynomial function in $b$.   As a simple example, consider the function with $\mathfrak{f}(x)=x^3$.  Then 
\[
T_2{\mathfrak f}(b)=b^3+3b^2(-b)+3b(-b)^2=b^3,
\]
which is a degree 3 polynomial in $b$ rather than the expected quadratic polynomial.
In fact, $T_n{\mathfrak f}(b)$ need not even be polynomial in $b$, as is seen by considering the function $f(x)=e^x$.  Compare this to Example \ref{e:0recoversF}.
\end{Observation}

The varying center tower was constructed with these three observations in mind, since defining the $n$th polynomial approximation for a functor $F:\AC\rightarrow \s$ means  letting $B$ vary.

 %%%%%%%%%%%%%%%%%%%%%%%%%%%%%%%%%%%%%%%%%%%%%%%%%%%%%%
\subsection{Varying Center Towers}
 %%%%%%%%%%%%%%%%%%%%%%%%%%%%%%%%%%%%%%%%%%%%%%%%%%%%%

%Defining the $n$th polynomial approximation for a functor $F:\AC\rightarrow \s$ means keeping $A$ fixed and letting 
%$B$ vary.  %Note that this is in stark contrast to the approach in Section 4 of \cite{G3}, where the classification of homogeneous functors is extended to functors from unbased spaces over $Y$ by letting the {\it initial} object vary, not the terminal object.  
In this section we will define a tower of functors $V_nF$ which act like an approximation tower for the object $F(A)$ in  the same way the $T_n{\mathfrak f}(b)$'s form a sequence of functions which approximates $f(0)$ as in Observation \ref{r:converge}.  As the terms $V_nF$ need not be polynomial nor approximations of the functor $F$, we will call this new sequence the {\it varying center  tower} (and not a Taylor tower) for $F:\AC\to \s$.

Let $\phi_f: \Cf \to \AC$ be the forgetful functor that sends $(X, \alpha_X, \beta_X)$ to $\alpha_X:A\rightarrow X$ (similar to $\phi$ from Section 4, \cite{G3}).  A functor  $F:\AC\to \s$ can be restricted to the functor $\phi^*_fF: \Cf\to \s$ defined by $\phi_f^*F = F\circ \phi_f $.  We often suppress $\phi^*_f$, and abuse notation by writing $F$ instead of $\phi^*_fF$ when the context is clear.
Let $A=(A, 1_A, f)$ denote the initial object of $\Cf$.

\begin{defn} \label{d:varying} The $n$th term in the varying center tower for the functor $F:\AC \to \s$ evaluated at the object $f:A\rightarrow B$  is $$V_nF(f: A\to B)\, := \, \Gamma_n^f(\phi_f^*F)(A).$$
\end{defn}

\begin{lem} \label{l:functorial} The $n$th term of the varying center tower, ${V_n}F$, is a functor from $\AC$ to $\s$ whenever $F:\AC\to \s$.
\end{lem}
\begin{proof} In order to show ${V_n}F$ is a functor, we need to define it for each morphism $\gamma$ from $f$ to $g$: 
\[ \xymatrix{& A \ar[dl]_f \ar[dr]^g \\ B\ar[rr]^{\gamma} && C.}\]  Let $\gamma_*:\Cf \to \Cg$ be the functor that is obtained by post-composition with $\gamma$.  That is, $\gamma_*(X, \alpha_X, \beta_X) = (X, \alpha_X, \gamma\circ\beta_X)$. 

Recall that  $\Gamma_n^f(\phi^*_fF)(A)$ is the homotopy cofiber of the map 
\[|(\perp_{n+1}^f)^{*+1}(\phi^*_fF(A))| \to \phi^*_fF(A).\]  
The morphism $\gamma:f\to g$ induces a morphism of the $(n+1)$-cubes   
that define $\ \perp_{n+1}^f(\phi^*_f F)$ and $\perp_{n+1}^g(\phi^*_gF)$.  This induces a map 
\[\gamma_*: (\perp_{n+1}^f)(\phi^*_f F)(A) \to (\perp_{n+1}^g)(\phi^*_g F)(A)\]
on  the total homotopy fibers of these cubes after the functor $F$ has been applied to them.     This in turn induces the map ${\Gamma^{\AC}_n}F(\gamma)$ of cofibers in the commuting diagram
\[ \xymatrix{ 
(\perp_{n+1}^f)^*(\phi^*_f F)(A) \ar[r]\ar[d] ^{\gamma}& F(A) \ar[r] \ar[d]^=& \Gamma_n^f(\phi^*_f F)(A) \ar[d]\\
(\perp_{n+1}^g)^*(\phi^*_g F)(A) \ar[r] & F(A) \ar[r] &\Gamma_n^g(\phi^*_g F)(A). }\]
Since ${V_n}F(f:A\to B) := \Gamma_n^f(\phi^*_f F)(A)$ and  ${V_n}F(g:A\to C):= \Gamma_n^g(\phi^*_g F)(A)$, this defines ${V_n}F(\gamma)$ .   The fact that $\gamma _*$ preserves compositions and identities on the underlying $(n+1)$-cubes ensures that ${V_n}F(\gamma)$ does as well.

\end{proof}

The functors $\Gamma_n^fF$ assemble into a tower of functors via natural transformations $q_n^f:\Gamma_n^fF \to \Gamma_{n-1}^fF$.  These natural transformations  can be used to assemble ${V_n}F$ into a tower of functors as well, justifying our use of the term ``varying center {\it tower}."

\begin{lem}\label{l:q-nat} There are natural transformations ${\rho_n}:{V_n}F \to {V_{n-1}}F$. \end{lem}

\begin{proof} 
For each object $f:A\rightarrow B$ of $\AC$, the map $q_n^f: {\Gamma^f_n}F\to {\Gamma^f_{n-1}}F$ induces a
natural transformation \[q_n^f:\Gamma_n^f\phi_f^*F \to \Gamma_{n-1}^f\phi_f^*F.\] 
   These assemble into a natural transformation $\rho_n:{V_n}F \to {V_{n-1}}F$ because any morphism $\gamma: f\to g$ in $\AC$ induces a commuting diagram
\[ \xymatrix{ \Gamma_n^f\phi_f^*F(A)\ar[r]^{\gamma_*}\ar[d]_{q_n^f} & \Gamma_n^g \phi_g^* F(A) \ar[d]^{q_n^g} \\
\Gamma_{n-1}^f \phi_f^* F(A) \ar[r]^{\gamma_*} & \Gamma_{n-1}^g \phi_g^* F(A)}.\]
To verify that this square commutes, use $\gamma_*$ to define maps between the cubes involved in the construction of $q_n^f$ and $q_n^g$.  See \S 5 of \cite{BJM} for details. 
\end{proof}

 Next we turn to understanding what role the notion of the degree of a functor plays in this context.    
 The first step is to define degree $n$ for functors of $\AC$.  We then establish the analogue of Proposition 3.5 for functors from $\AC$.
 
 \begin{defn}
 A functor $F:\AC\rightarrow \s$ is degree $n$ relative to $A$ provided that for all objects $f:A\rightarrow B$ in 
 $\AC$, $\phi_f^*F: \Cf\to \s$ is degree $n$ relative to $f$ (see Definition \ref{d:rel-f}). 
 \end{defn}
 
 \begin{prop}\label{p:degexc}The functor  $F:\AC \to \s$ is degree $n$ relative to $A$ if and only if
 $F$ is $n$-excisive relative to $A$.  If $F$ commutes with realizations, then $F$ is degree $n$ relative to $A$ if and only if $F$ is $n$-excisive.
 \end{prop}
 
 \begin{proof}  Let $\emptyset \subset {\bf n}$ denote the map in $P({\bf n})$ which is the inclusion of the empty set into the set $\bf n$.
For  any cubical diagram $\chi:P({\bf n}) \to \AC$, the choice $f_\chi=\chi(\emptyset\subset {\bf n})\circ \alpha_{\chi(\emptyset)}$, 
\[\xymatrix{A\ar[r]^{\alpha_{\chi({\emptyset})}\quad} & \chi(\emptyset)\ar[r]^{\chi(\emptyset\subset {\bf n})} & \chi({\bf n})},\]
 provides us with an object $f_\chi$ of $\AC$ with the property that the cubical diagram $\chi$ can be viewed as a cubical diagram in $\C _{f_{\chi}}$.
  
Suppose that $F$ is degree $n$ relative to $A$.   Let $\chi$ be any strongly cocartesian $(n+1)$-cubical diagram 
in $\AC$ with initial object $A$.   By assumption,  $F$ is degree $n$ relative to $f_{\chi}$ and $\chi$ is an $(n+1)$-cubical diagram in $\C _{f_{\chi}}$. By Proposition \ref{p:deg-exc}, $F$ takes 
$\chi $ to a cartesian diagram, and so $F$ is $n$-excisive relative to $A$.   

Conversely, if $F$ is $n$-excisive relative to $A$ as a functor of $\AC$, then by restriction, it 
is $n$-excisive relative to $A$ as a functor of $\Cf$ for any $f:A\rightarrow B$ in $\C$.   Proposition \ref{p:deg-exc} guarantees that $F$ is degree $n$ relative to $f$.   

The proof
of the second part of the theorem is similar.   
 
 \end{proof} 
 As is the case for Taylor series of functions (cf. Remark \ref{r:notdegreen}),
  we do not expect ${V_n}F$ to be a degree $n$ functor of $\AC$.    The next example demonstrates that $V_nF$ is not necessarily degree $n$,  even in degree 0.
  \begin{ex} \label{e:0recoversF} Let $F:\AC\to \s$.  Then $V_0F(f:A\to B)= \Gamma_0^fF(A)$.  The latter is calculated using the first cross effect functor
  \[ cr_1^fF(X) := \hofib \left(F(X)\to F(B)\right)\]
  for any $X\in \Cf$, including $X=A$.  Since $cr_1^nF(X)\simeq cr_1F(X)$ for all\\  $n\geq 1$, the degree 0 approximation $\Gamma_0^fF(A)$ is defined by taking the homotopy cofiber of \[ cr_1F(A) \to F(A).\]  But the homotopy fiber sequence
  \[ cr_1F(A) \to F(A) \to F(B)\]
  defining $cr_1F$ is also a cofiber sequence in sequence in spectra, hence $\Gamma_0^fF(A)=F(B)$.  Thus,
  \[ V_0F(f:A\to B) = F(B)\]
  for all $A\to B$ in $\AC$.  But this functor is not degree 0 as a functor of $B$, i.e.,  it is not constant if $F$ is not constant.  In fact, it will not even have finite degree if $F$ is not finite degree.
  \end{ex}
  
  However, even though $V_nF$ may not have finite degree, starting with a degree $n$ functor ensures that the varying center tower 
 \[\dots \rightarrow {V _n}F\rightarrow  V _{n-1}F\rightarrow \dots \rightarrow {V _0F}\]
 converges to the constant functor $F_A$ (i.e., given a functor $F:\AC\to \s$, $F_A:\AC\to \s$ is  the functor  satisfying $F_A(f:A\to B) = F(id:A=A)$ for all $f:A\to B$ in $\AC$).   
 \begin{prop}\label{p:constantconverge}
 Let $F:\AC\rightarrow \s$.   If $F$ is degree $n$ relative to $A$, then the natural transformation $\phi_f^*F\to \Gamma_n^f(\phi_f^*F)$ evaluated at $A$ induces a weak equivalence $F_A \simeq V_kF$ for all $k\geq n$.  
 \end{prop}
 \begin{proof}   Consider $f:A\rightarrow B$.   By assumption, $\phi^*_fF$ is degree $n$ relative to $f$.  
 By Proposition 3.22 of \cite{BJM}, $\phi^*_fF$ is degree $k$ relative to $f$ for all $k\geq n$.   Then, by Proposition 5.6 of \cite {BJM}, $\Gamma _k^f\phi^*_fF\simeq \phi^*_fF$, and so $ V_kF(f:A\rightarrow B)=\Gamma _k^f\phi^*_fF(A)\simeq \phi^*_fF(A)=F(A)$.      
 \end{proof}

This allows us to consider what happens as $n$ increases.

\begin{defn} Let $F:\AC \to \s$.  We say that the tower $\{{V_n}F(f)\}$ {\it converges at $f$} if the constant functor $F_A$ is homotopy equivalent to the limit ${V_{\infty}}F(f):= \holim_n {V_n}F(f)$.
\end{defn}

The preceding proposition tells us that if $F$ is degree $n$ relative to $A$, then ${V_n}{F}$ is equivalent to the constant functor with value $F(A)$  for all $f:A\to B$ in $\AC$.  Thus, the  calculus tower $\{{V_n}F\}$  is trying to approximate the value of $F$ at the initial object $A$, just as the Taylor polynomials approximated the initial value $f(0)$ when we allowed the center to become the variable in Remark \ref{r:converge}.   This is a departure from the usual:  the Taylor towers of \cite{G3}, \cite{R&B} and \cite{BJM} for functors $F$, under certain conditions, can be treated as approximations   to the {\it functor} $F$, rather than a constant functor given by a particular value of $F$.   For functors $G$ from the category of unbased topological spaces, this means that ${V_n}G$ is trying to approximate the value of $G$ on the {\it empty set}, the initial object in the category of unbased spaces.

%%%%%%%%%%%%%%%%%%%%%%%%%%%%%%%%%%%%%
\subsection{The de Rham complex as a varying center tower}
%%%%%%%%%%%%%%%%%%%%%%%%%%%%%%%%%%%%%

We finish this section by justifying the claim made in the introduction that  for rational algebras, the de Rham complex is the varying center tower for the forgetful functor from rational algebras to modules.   We begin by describing the functors and categories we use to do so.  

We use $\CQ$ to denote the category of commutative rational algebras and $s_\cdot \CQ$ to denote the category of simplicial objects in $\CQ$.  In particular, let $\Q_\cdot$ denote the constant simplicial object in $s_\cdot \CQ$ that is $\Q$ in each simplicial.  Let ${\U}$ denote the forgetful functor from $\CQ$ to the category of ${\Q}$-modules.  We can extend this to a functor from $s_\cdot \CQ$ to simplicial ${\Q}$-modules by applying $\U$ degreewise.  

Recall that for a morphism of rational algebras $f:X\rightarrow B$, the de Rham complex $DR_XB$ is the cochain complex of exterior algebras of the K\"ahler differentials:
\[ \dots \leftarrow \Omega^3_{B/X}\leftarrow \Omega^2_{B/X}\leftarrow\Omega_{B/X}\leftarrow B. \]
See \cite[Sections 8.8.1 and 9.8.9]{weibel} for further details.  The construction is natural in maps $f:X\to B$, so given a map of simplicial algebras $f:X_\cdot \to B_\cdot$ we can construct a simplicial cochain complex whose $k$th object is $DR_{X_k}B_k$.  Let
 $DR_{X_\cdot}B_\cdot$ denote the associated (second quadrant) bicomplex obtained via normalization.  
 
 For a fixed $B_\cdot$ in $s.\CQ$, the functor $DR_{(-)}B_{\cdot}:(X_\cdot \rightarrow B_\cdot)\mapsto DR_ {X_\cdot} B_\cdot$ is a functor whose source category is the  category of simplicial rational algebras over $B_{\cdot}$.  However $DR_{(-)}B_{\cdot}$ is not in general  a homotopy functor; to remedy this we will assume that the map $X_{\cdot}\to B_{\cdot}$ is a cofibration in $s.\CQ$. 
 The total complex ${\rm Tot}^{\Pi}(DR_{X_{\cdot}}B_{\cdot})$ is what we mean by the de Rham complex of $f:X_{\cdot}\rightarrow B_{\cdot}$, as the columns of  $DR_{X_{\cdot}}B_{\cdot}$ correspond to the exterior algebra terms in the de Rham complex and the filtration of $DR_{X_\cdot}B_\cdot$  by columns converges to ${\rm Tot}^{\Pi}DR_{X_{\cdot}}B_{\cdot}$.  However, for ease of exposition, we will usually suppress  ${\rm Tot}^{\Pi}$ and work directly with the underlying bicomplex.  
 %(see [REF] for the model structure of $\sComm$)%    

 For a simplicial rational algebra over $B_{\cdot}$,  $X_{\cdot}\to B_{\cdot}$, we use $DR^n_{X_\cdot}B_\cdot$ to denote the $n$th truncated de Rham complex of $X_\cdot\rightarrow B_\cdot$, i.e., the bicomplex whose first $n+1$ columns are the same as the first $n+1$ columns of $DR_{X_\cdot}B_\cdot$, and whose columns are identically 0 thereafter:
 \[
 DR^n_{X_{\cdot}}B_{\cdot}=\dots \leftarrow 0\leftarrow\dots\leftarrow 0\leftarrow \Omega^n_{X_{\cdot}}B_{\cdot}\leftarrow\dots\leftarrow \Omega^2_{X_{\cdot}}B_{\cdot}\leftarrow  \Omega^1_{X_{\cdot}}B_{\cdot}\leftarrow B_{\cdot}.
 \]

 For a fixed rational algebra, $f: \Q_{\cdot}\to B_{\cdot}$, we make use of the $n$th truncated de Rham complex as a functor of $(s_{\cdot} {\CQ})_f$ as follows.
 
 \begin{defn}\label{d:drn}
 Let $f: \Q_{\cdot}\to B_{\cdot}$ be a morphism in $({s_\cdot \CQ})_f$.  The functor $DR_n^f$ is a functor from  $({s_\cdot \CQ})_f$ to the category of rational chain complexes  that takes the object $X_\cdot=\Q_\cdot\rightarrow X_\cdot\rightarrow B_\cdot$ to 
 \[
 DR_n^f(X_{\cdot}):={\rm Tot}^{\Pi}(DR^n_{X_{\cdot}}B_{\cdot}).
 \] 
  \end{defn}
 Our goal is to outline a proof of  the following unpublished result of Goodwillie and Waldhausen.  We learned of this result and method of proof from conversations with Goodwillie.  
\begin{prop}\label{o:derham}  For a  rational algebra $\Q \rightarrow B$, and its cofibrant replacement $\Q_\cdot\rightarrow B_\cdot$, 
$V_n\U(\Q_\cdot\rightarrow B_\cdot)\simeq DR^n_{\Q_\cdot}B_\cdot.$
\end{prop}

Our first step in justifying this claim is to describe a strategy for identifying  $V_nF$ for a functor $F:{\AC}\rightarrow \s.$  By Definition \ref{d:varying}, to identify $V_nF(f:A\rightarrow B)$ for  $F:{\AC}\rightarrow \s$ and an object $f:A\rightarrow B$ in $\AC$,  we must  find  $\Gamma_n^fF$  and then determine the value of this functor at $A$ (viewed as the object $A=A\rightarrow B$ in $\Cf$).  When $F$ commutes  with realizations, Proposition \ref{p:degexc} and a variant of Proposition 1.6 of \cite{G3} provide a means of  proving that a particular functor is equivalent to $\Gamma_n^fF$.  More explicitly, these results guarantee that we can determine $V_nF$  by first identifying  for each $f:A\rightarrow B$ in $\AC$ a functor $G_n^f:\Cf\rightarrow \s$ that is natural in $f$ and satisfies the following:
\begin{enumerate}
\item $G_n^f:\Cf\rightarrow \s$  is $n$-excisive, and

\item there is a natural transformation $\phi_f^*F\rightarrow G_n^f$ with the property that there are constants $\kappa$ and $c$ such that for any object $A\rightarrow X\rightarrow B$ in $\Cf$, where $X\rightarrow B$ is $k$-connected  with $k\geq \kappa$, the induced morphism 
$\phi_f^*F(X)\rightarrow G_n^f(X)$ is at least $(-c+(n+1)k)$-connected.

\end{enumerate}
These conditions guarantee that $G_n^f\simeq \Gamma_n^fF$ as a functor of $\Cf$.  Then  $V_nF(f:A\rightarrow B)\simeq G_n^fF(A)$.  

We apply this result to the case where $\C=s.\CQ$, $A=\Q_{\cdot}$, and $F={\mathcal U}$, the forgetful functor of Proposition \ref{o:derham}. We claim that in this context, the functor $DR^f_n$ of Definition \ref{d:drn} is the correct choice for the functor $G_n^f$ described above.  
Thus, to confirm Proposition \ref{o:derham}, it suffices to show that for any cofibrant object $f:\Q_\cdot\rightarrow B_\cdot$ in $s.\CQ$,  
\begin{enumerate}
\item $DR_n^f$ is $n$-excisive as a functor of $(s.\CQ)_f$, and
\item  for an object $X_\cdot=\Q_\cdot\rightarrow X_\cdot \rightarrow B_\cdot$ in $(s.\CQ)_f$, where $X_\cdot\rightarrow B_\cdot$ is $k$-connected with $k\geq 1$, 
$$
\U(X_\cdot)\rightarrow DR^f_n{X_\cdot}$$ is at least $(n+1)k-(n+1)$-connected.  The natural map $\U(X_\cdot)\rightarrow DR^f_n(X_\cdot)$ is the map that is $f:X_{\cdot}\rightarrow B_{\cdot}$ in the $0$th level of the complex and $0$ elsewhere.  
\end{enumerate}
We describe how to do this in what follows.

%As we are primarily concerned with what happens when $B_\cdot$ is a cofibrant replacement for a rational algebra $B$, we will describe how to establish (1) and (2) for polynomial algebras.  

We begin by explaining why condition (1) holds.  In the case $n=1$, for  morphisms of commutative rational algebras  $X\rightarrow B$, one can verify  that the functor $(X\rightarrow B)\mapsto \Omega^1_{B/X}$ is $1$-excisive and reduced by first recalling that $\Omega^1_{B/X}$ is isomorphic to $I/I^2$ where $I$ is the kernel of the map $B\otimes _XB\rightarrow B$ (\cite{weibel}, 9.2.4). The functor $I/I^2$ is equal to the composition $(K/K^2)\circ E$ where $E$ is the functor from $\CQ$ to the category of augmented rational algebras that takes $\Q\rightarrow X\rightarrow B$ to $B\rightarrow B\otimes _XB\rightarrow B$ and $K(B\rightarrow Y\rightarrow B)$ is the augmentation ideal functor from  augmented rational algebras to rational modules.  The functor $E$ preserves cocartesian diagrams and the functor $K/K^2$ is known to be linear (see for example, \cite{R&B}, \cite{KM}, \cite{HH}.)  This implies that $DR^f_1$ is degree 1.  

For $n>1$, the fact that $(X\rightarrow B)\mapsto \Omega^1_{B/X}$ is $1$-excisive and reduced  also tells us that   $(X\rightarrow B)\mapsto \Omega^n_{B/X}$ is a homogeneous degree $n$ functor.   In particular, this holds because $\Omega^n_{B/X}$ is the $n$-fold exterior power of $ \Omega^1_{B/X}$.  As such it is given by the orbits of the canonical action of the $n$th symmetric group $\Sigma_n$ (and hence, homotopy orbits, since we are working rationally) of a multilinear functor of $n$ variables.  By fundamental results of Goodwillie (see \cite{G3} or \cite{K}), we know that functors of this form are homogeneous of degree $n$.  As a result, the functor $DR^f_n$ is degree $n$.  

To see that (2) holds, consider an object in $(s.\CQ)_f$, $\Q_\cdot\rightarrow X_\cdot\rightarrow B_\cdot$, where $X_\cdot\rightarrow B_\cdot$ is $k$-connected for some $k\geq 1$.  Consider the map of bicomplexes $tr_n: DR_{X_\cdot}B_\cdot \to DR^n_{X_\cdot}B_\cdot$   which truncates the de Rham complex at the $n$th column. 
We note that any $k$-connected cofibration $f$ has a factorization
\[ \xymatrix{X_\cdot\ar[r]^{\tilde{f}} & B_\cdot' \ar[r]^{\simeq} & B_\cdot}\]
where the map $\tilde{f}$ is a cofibration, an isomorphism in dimensions up to $k$ and an injection in dimension $k+1$.  Thus, since $DR_{(-)}B_\cdot$ and $DR^n_{(-)}B_\cdot$ are homotopy functors, we can assume that our $k$-connected cofibration has the same form as ${\tilde {f}}$.  As a result, $\Omega^1_{B_i/X_i}$ is zero for $0\leq i\leq k$, and in turn $\Omega^m_{B_i/X_i}$ is zero for $0\leq i\leq mk$ (using the fact that $\Omega^n_{B_\cdot/X_\cdot}$ is the $n$th exterior algebra of $\Omega^1_{B_\cdot/X_\cdot}$).   %Similarly, $DR^n_{X_\cdot}B_\cdot$ vanishes under the same line of slope $k$.  This implies that the truncation map $tr_n$ is an isomorphism under this vanishing line.  
Taking the total complex, we see that this means that 
$tr_n:DR_{X_\cdot}B_{\cdot}\rightarrow DR^n_{X_\cdot}B_\cdot$ 
is $(n+1)k-(n+1)$-connected.  

We next claim that because $k\geq 1$, ${\rm Tot}^{\Pi}(DR_{X_{\cdot}}B_{\cdot})\simeq X_{\cdot}$.  This follows from the fact that when $k\geq 1$, the resulting connectivity of the columns of $DR_{X_{\cdot}}B_{\cdot}$ guarantees that the bicomplex is bounded, and so ${\rm Tot}^{\Pi}(DR_{X_{\cdot}}B_{\cdot})\simeq {\rm Tot}^{\oplus}(DR_{X_{\cdot}}B_{\cdot})$.  Since we are  working rationally, the Poincar\'e lemma tells us that the $m$th row of $DR_{X_{\cdot}}B_{\cdot}$ is equivalent to $X_m$ and so ${\rm Tot}^{\oplus}(DR_{X_{\cdot}}B_{\cdot})\simeq X_{\cdot}$.  Hence $\U(X_{\cdot}\rightarrow B_{\cdot})\rightarrow  DR^f_n(X_{\cdot}B_{\cdot})$ is at least $(n+1)k-(n+1)$-connected.  

%Then, because $k\geq 1$ and we are working rationally, the Poincar\'e lemma guarantees that each row of  $DR_{X_\cdot}B_\cdot$ is equivalent to $X_\cdot$.  Hence, $DR_{X\cdot}B_\cdot\simeq X_{\cdot}$ so that the natural transformation $\U(X_\cdot)\rightarrow DR^n_{X\cdot}B_{\cdot}$ is at least $(n+1)k-(n+1)$-connected.  

Given that conditions (1) and (2) hold, we know that for $f:\Q_\cdot\rightarrow B_\cdot$, $\Gamma_n^f\U\simeq DR_n^f$.  Evaluating at the initial object $\Q_\cdot$, that is, $\Q_\cdot =\Q_\cdot\rightarrow B_\cdot,$ in $(s.\CQ)_f$ gives us
\[
V_n\U(f:\Q_\cdot\rightarrow B_\cdot)\simeq DR_n^f(\Q_\cdot)=DR^n_{\Q_{\cdot}}B_{\cdot}, 
\]
as predicted by Proposition \ref{o:derham}.  The convergence of this tower when $\Q_\cdot\rightarrow B_{\cdot}$ is a $k$-connected cofibration with $k\geq 1$ (which will be addressed further in Corollary 4.7) is now a restatement of the Poincar\'e Lemma.  That is, $DR_{\Q}B_\cdot \simeq \Q \simeq V_\infty \U (\Q\to B_\cdot)$ in this case.  A striking feature of the varying center tower is that  $V_\infty \U$ is independent of $B_\cdot$ entirely when the tower converges.  
%That is, the varying center tower $V_n \U$ bundles together all of the towers $\Gamma_n^f \U$ for every $f:\Q\to B_\cdot$ to provide a filtration of a single object, $\Q=\U(\Q)$.

%%%%%%%%%%%%%%%%%%%%%%%%%%%%%%%%%%%%%%%%%%%
\section{Proof of  main theorem \label{s:proof}}
%%%%%%%%%%%%%%%%%%%%%%%%%%%%%%%%%%%%%%%%%%%%

The goal of this section is to prove Theorem \ref{t:infinity}, which gives an equivalence between the limit of the varying center tower,
$ V_\infty F(f:A\to B)$,
and the total space of the cosimplicial spectrum
$ |\Gamma_\infty^f  F(X\otimes_A \sk_n \Delta^\bullet_*)|$ for any $A\rightarrow X\rightarrow B$ that factors $f$.
We  proceed inductively, first proving results for  linear functors, then finite degree functors and then for limits of Taylor towers.

The advantage of examining $\Gamma_\infty^f F(X\otimes_A \sk_n \Delta^\bullet_*)$ is that  $X\otimes_A \sk_n \Delta$ can be thought of as being more highly connected than $X$.  
Combined with a good notion of analyticity (Definition \ref{d:rho})  this allows us to describe $V_\infty F(f:A\to B)$ more succinctly as $\Tot | F(B\otimes_A \sk_n \Delta^\bullet_*)|$, avoiding the need to calculate the Taylor tower $\{\Gamma_n^f F\}$ at all.  For spaces, where the notion of  ``connectivity" is well-understood, this is made precise in Theorem \ref{t:connected}.  The general form of this result for arbitrary model categories is stated in Corollary \ref{p:analyticprep}.

Let $f:A\to B$ be any object of $\AC$.  Let $A\to X\to B$ be any factorization of $f$, so that $X=A\rightarrow X\rightarrow B$ is an object of $\Cf$.    
In this section we work with functors $F:{\sc D}\rightarrow \s$ where ${\sc D}$ is $\AC$ or $\Cf$.  When starting with $F:\AC\rightarrow \s$, we will also use $F$ to represent the functor $\phi_f^*F:\Cf\rightarrow \s$, as defined in Section 3.3. As in the previous section, we assume that $F$ is a homotopy functor and $\C$ is a simplicial model category.   When applying $F$ to simplicial or cosimplicial objects in ${\sc D}$, $F$ will be applied degreewise.   In Section 3, we defined $B\otimes_A U$ for  a morphism $f:A\to B$ in $\Cf$ and a finite set $U$;  it is the colimit of $\#U$ copies of $B$ ``added" along $A$ via the maps $f$.  This can be generalized to cosimplicial-simplicial sets $Z$ by using $Z_k^n$ in place of $U$.  The various face, degeneracy, coface and codegeneracy maps involved become insertions and fold maps.    Throughout this section we write $\emptyset$ for the constant cosimplicial-simplicial set that is empty in each simplicial and cosimplicial degree.  Thus $A= X\otimes_A \emptyset$ is a constant cosimplicial simplicial object of $\C$.  The inclusion map $\emptyset \to \sk_n\Delta_*^\bullet$ induces a map of cosimplicial simplicial 
$\C$-objects $A\to X\otimes_A \sk_n\Delta_*^\bullet$, where $A$ denotes the constant cosimplicial simplicial object.
This map induces a weak equivalence of spectra when a functor $F$ of finite degree  is applied to it, as is proved in the next few propositions.

\begin{prop} \label{p:linear} If $F:\C_f \to \s$ is degree $1$ relative to $f$ then for all $X$ in $\Cf$, there is a weak homotopy equivalence
\[F(A)\overset{\simeq\ }{\rightarrow} Tot^m|\ F(X\otimes_A sk_n \Delta_*^\bullet)|\]
 for all $m\geq n+1$.    
Thus, if $F:\AC \rightarrow \s$ is degree $1$ relative to $A$, then for all $f:A\rightarrow B$ in $\AC$ and $A\rightarrow X\rightarrow B$ in $\Cf$, there is a pro-equivalence of towers
\[\tower F(A)\to \tower |F(X\otimes_A \sk_n\Delta_*^\bullet)|.\]
\end{prop}

When $F$ commutes with realizations, the case $m=n+1$ can be deduced from Proposition 3.0.3 of \cite{Rosona}, using the fact that in this case, $F(A)\simeq T_1F(A)$.  However the results of \cite{Rosona} do not extend to a pro-equivalence of towers.

\begin{proof}
 We claim that $F(X\otimes_A U) \simeq F(X)\otimes_{F(A)} U$ for any finite non-empty set $U$.      First note that if $U$ has exactly one element, then 
$F(X\otimes _A U)\simeq F(X)\simeq F(X)\otimes _{F(A)}U$.   For $U={\bf n}$, consider the $n$-cube ${X\otimes _A}$ with $(X\otimes _A)(T)=X\otimes _A T$ for subsets $T\subseteq U$. This is a strongly cocartesian  $n$-cube.  By Proposition 3.22 of \cite{BJM} and Proposition \ref{p:deg-exc}, $F$ is  $k$-excisive relative to $A$ for all $k\geq 1$.   As a result, $F(X\otimes_A)$ is a cartesian $n$-cube.   Since $F$ takes values in spectra, $F(X\otimes_A)$ is also cocartesian.   By induction, for each $T\subseteq {\bf n}$ with $T\neq {\bf n}$, $F(X\otimes_AT)\simeq F(X)\otimes_{F(A)}T$.   Since $F(X\otimes_A)$ is cocartesian, this implies that 
\begin{equation}\label{e:ind}\begin{aligned}
F(X\otimes _A {\bf n})&\simeq {\rm hocolim}_{T\subseteq {\bf n}\atop T\neq {\bf n}}F(X\otimes _AT)\\
&\simeq 
 \operatorname{ hocolim}_{T\subseteq {\bf n}\atop T\neq {\bf n}} F(X)\otimes _{F(A)} T\\
 &\simeq F(X)\otimes _{F(A)}{\bf n}.\end{aligned}
\end{equation}

 Now,  $F(A)\simeq F(X)\otimes_{F(A)} \emptyset$ as constant cosimplicial simplicial spectra.  So we have equivalences of cosimplicial simplicial spectra
\[ F(A) \simeq F(X)\otimes_{F(A)} \emptyset = F(X)\otimes_{F(A)} \Tot^m \sk_n \Delta^\bullet_*\]
whenever $m\geq n+1$, by Proposition \ref{p:Tot}.  
Suppose that $F(A)=\star$, the base point in the category $\s$.  In $\s$, the coproduct $F(X)\otimes U = \coprod_U F(X)$ is weakly equivalent to a product.  The functor $\Tot^m$ commutes with products, so we have an equivalence of simplicial spectra
\[ F(A)=\star \simeq \Tot^m( F(X)\otimes \sk_n \Delta^\bullet_*)\]
where the left hand side is a constant simplicial spectrum.

Upon taking the geometric realization, we obtain the equivalence
  \[\star\simeq |\Tot^m(F(X)\otimes \sk_n \Delta^\bullet_*)|\simeq \Tot^m(|F(X)\otimes \sk_n \Delta_*^\bullet|)\]
  since geometric realizations commute with finite homotopy limits such as $\Tot^m$ in $\s$.
Since $F$ is $1$-excisive relative to $A$, this is equivalent to $\Tot^m |F(X\otimes \sk_n \Delta_*^\bullet)|$ by (\ref{e:ind}).
Since this holds for every $m\geq n+1$, we obtain the desired pro-equivalence 
\[ \tower \star \simeq \tower |F(X \otimes \sk_n \Delta_*^\bullet)|.\]

Now, if $F(A)\neq \star$, form the reduced functor $\widetilde{F}$ by 
\[ \widetilde{F}(X) := \operatorname{hocofiber}\left( F(A)\to F(X) \right).\]
Note that if $F$ is degree 1 relative to $f$, then so is $\widetilde{F}$.   Furthermore, $\widetilde{F}(A)\simeq \star$.  Since geometric realization and the functor $\Tot^m$ preserve (co)fibration sequences of spectra, we have a  (co)fibration sequence 
\[ F(A) \to \Tot^m | F(X\otimes_A \sk_n \Delta^\bullet)| \to \Tot^m | \widetilde{F}(X\otimes_A \sk_n \Delta^\bullet)|.\]
By the previous case, $\Tot^m |\widetilde{F}(X\otimes_A \sk_n \Delta^\bullet)|\simeq \star$ whenever $m\geq n+1$.  The result follows.
\end{proof}

It is also possible to prove Proposition \ref{p:linear} by constructing explicit cosimplicial homotopies  to show that ${\Tot}|F(X\otimes_A \sk_n\Delta^{\bullet}_*)|\simeq \star$ in the case that $F(A)\simeq \star$.   

The conclusion of Proposition \ref{p:linear} can be reformulated as a statement about the coskeleta of  the cosimplicial simplicial spectrum $F(X\otimes_A \sk_n\Delta^\bullet_*)$.
 Let $Z$ be any cosimplicial object.  The $k$th matching  object of $Z$ is defined by
\[ M^k Z:\ = \ \lim_{\alpha:[k+1]\to [t]} X^t \]
where $\alpha:[k+1]\to [t]$ is a surjection in $\bf \Delta$ with $t\leq k$.  See \cite[\S VII.4]{GJ}  for details.  In particular, this means that 
\[ (\csk^n Z)^k = \begin{cases} Z^k & k\leq n \\
M^{k-1}(Z\circ i_n) & k>n\end{cases} \]
where  the restriction $Z\circ i_n$ is as discussed in Definition 2.2.  Note that, in particular, $(\csk^nZ)^{n+1}=M^nZ$.
\begin{cor}\label{l:csk}  If $F$ is degree 1 relative to $f$, then for all $X$ in $\Cf$ and all $m\geq n+1$, we have a levelwise equivalence of cosimplicial simplicial spectra \[ F(X\otimes_A \sk_n \Delta_*^\bullet)\simeq \csk^m F(X\otimes_A \sk_n \Delta_*^\bullet).\]
\end{cor}

\begin{proof}  For any cosimplicial spectrum and any $m\geq 1$, there is a pullback square
\[ \xymatrix{ \Tot^m Z \ar[rr] \ar[d] && \hom(\Delta^m, Z^{m}) \ar[d] \\
\Tot^{m-1} Z \ar[rr] && \hom(\partial \Delta^m, Z^m)\times_{\hom(\partial \Delta^m, M^{m-1}Z)} \hom(\Delta^m, M^{m-1}Z)}
\]
where $\hom$ is the enriched $\hom$ in spectra.
From this pullback square, it is possible to show that
\[ \hofib (\Tot^mZ \to \Tot^{m-1}Z )\simeq \Omega^m \hofib (Z^m \to M^{m-1}Z)\]
(see (\cite[\S VIII.1]{GJ}), or  \cite{tot-primer} for full details).  When $Z=F(X\otimes_A \sk_n \Delta_*^\bullet)$, Proposition \ref{p:linear}  implies that $\Tot^m Z \to \Tot^{m-1} Z$ is an equivalence for all $m>n+1$.  Since $F$ is a functor to spectra, we can conclude that $Z^m \to M^{m-1}Z$ is an equivalence for all $m>n+1$ as well.  %By the definition of $\csk^m$, the conclusion now follows.

The rest of the proof follows by induction.  %When $m=n+2$, we have established that $Z^{n+2}\to M^{n+1}Z=(\csk^{n+1}Z)^{n+2}$ is an equivalence.  By definition of the coskeleton, we have that $(\csk^kZ)^{n+2}=Z^{n+2}$ for every $k>n+1$.  Thus for all $k\geq n+1$, we have $(\csk^kZ)^{n+2}=Z^{n+2}$. 

%Now assume that for $2\leq t< N$ we have $(\csk^k Z)^{n+k}=Z^{n+t}$ for all $k\geq n+1$.  When $t=N$, we have $(\csk^k Z)^{n+N}=Z^{n+N}$ by definition of $\csk^k$ for all $k\geq n+N$.  The case $k=n+N-1$ follows from the equivalence $Z^{n+N}\to (\csk^{n+N-1}Z)^{n+N}$ of the preceding paragraph.  If $n+1\leq m<n+N-1$, then
%\[ (\csk^m Z)^{n+N} = \lim_{\alpha:[n+N]\to [t]} Z^t\]
%where $\alpha$ is a surjection and $t\leq m$.  By inductive hypothesis, we already have for each $n+1\leq m<n+N-1$, \[Z^m= \lim_{\beta:[m]\to [t]} Z^t\] where $\beta$ is a surjection and $t\leq m$.  Thus, by cofinality, $(\csk^m Z)^{n+N} = \lim_{\alpha:[n+N]\to [t]} Z^t$ where $t\leq n+1$.  That is, $(\csk^m Z)^{n+N} = (\csk^{n+1} Z)^{n+N}$.
\end{proof}

We next consider the case of a functor that is  degree $k$ relative to $f$.  Recall that under the hypothesis that $F$ commutes with  realizations, this is the same as saying that $F$ is $k$-excisive (Proposition \ref{p:deg-exc}).

\begin{prop}\label{p:degree_n} Let $F:\AC \to \s$ be a functor that commutes with realizations.  Let $f:A\rightarrow B$ and   let $X$ be any object in $\Cf$.  For all $k\geq 1$ and $n\geq 0$, the map 
\[ V_kF(f)=\Gamma^f_k F(A) \to \Tot^t|\Gamma_k^fF(X\otimes_A \sk_n\Delta_*^\bullet)|\]
is an equivalence of spectra for each $t\geq (n+1)k$.
Thus there is a pro-equivalence of towers 
\[\tower V_kF(f) \to \tower |\Gamma_k^fF(X\otimes_A \sk_n\Delta_*^\bullet)|.\]
\end{prop}

\begin{proof}
We establish first that we have a pro-equivalence of cosimplicial spectra
\[ \tower\widetilde{\Gamma}_k^fF(A)\to \tower\ |\widetilde{\Gamma}_k^fF(X\otimes_A \sk_n \Delta_*^\bullet)|\]
where $\widetilde{\Gamma}_k^f F$ is the fiber of the natural transformation $q_k^f: \Gamma_k^fF\to \Gamma_{k-1}^fF$.
Let $\widetilde{L}^f_kF(-, \ldots , -)$ be the multilinear functor of $k$ variables associated to $\widetilde{\Gamma}_k^fF$ by Proposition \ref{p:layers}.  Denote the evaluation on the diagonal, $\widetilde{L}^f_kF(X, \ldots , X)$, by $\widetilde{L}^f_kF(X)$.
Since $F$ commutes with  realizations, Proposition \ref{p:layers} guarantees that $\widetilde{\Gamma}_k^fF(X) \simeq \widetilde{L}^f_kF(X)_{h\Sigma_k}$.

Fix objects $X_1, \ldots , X_{k-1}$ for the first $k-1$ variables of $\widetilde{L}^f_kF$ and consider the single variable functor $\widetilde{L}^f_kF(X_1, \ldots , X_{k-1}, -)$.
Since $\widetilde{L}^f_kF$ is linear in each variable, Proposition \ref{p:linear} implies that we have a weak equivalence of spectra
\[ \widetilde{L}^f_kF(X_1, \ldots , X_{k-1}, A) \simeq \Tot^m |\widetilde{L}^f_kF(X_1, \ldots, X_{k-1}, X\otimes_A \sk_n\Delta_*^\bullet)|\]
for each $X_i\in \Cf$ and each $m\geq n+1$.
We repeat this process in each variable separately to obtain a similar equivalence for each of the $k$ variables.  By Corollary \ref{l:csk} and Proposition \ref{p:fix}, these equivalences assemble to 
\[ \widetilde{L}^f_kF(A) \simeq \Tot^{mk} |\widetilde{L}^f_kF(X\otimes_A \sk_n\Delta_*^\bullet)|\]
for any $m\geq n+1$ and any $X\in \Cf$.  In the category $\s$, finite homotopy limits commute with finite homotopy colimits.  The partial totalization $\Tot^{mk}$ is a finite homotopy limit  so it commutes with the homotopy colimit that constructs the homotopy orbits of the $\Sigma_k$-action in spectra.   We have
\begin{align*} 
\widetilde{L}^f_kF(A)_{h\Sigma_k} &\simeq \left( \Tot^{mk}|\widetilde{L}^f_kF(X\otimes_A \sk_n\Delta_*^\bullet)|\right)_{h\Sigma_k}\\  
&\simeq \Tot^{mk} \left( |\widetilde{L}^f_k F(X\otimes_A \sk_n\Delta_*^\bullet)|_{h\Sigma_k} \right) \\ 
& \simeq \Tot^{mk}| (\widetilde{L}^f_k F(X\otimes_A \sk_n \Delta_*^\bullet))_{h\Sigma_k} |
\end{align*}
whenever $m\geq n+1$.  The last equivalence follows from the fact that homotopy colimits commute.  
Thus, we have a pro-equivalence 
\[\tower \widetilde{\Gamma}^f_kF(A)\to \tower |\widetilde{\Gamma}^f_kF(X\otimes_A \sk_n\Delta_*^\bullet)|\] for all $k$ and $n$ greater than or equal to $1$.

The proof of the proposition now follows by induction on $k$, using the (objectwise) fibration sequence of functors
\[ \widetilde{\Gamma}_k^fF \to \Gamma_k^fF \to \Gamma_{k-1}^fF.\]   
The base case is given by Proposition \ref{p:linear} since  $\Gamma_1^fF$ is degree 1.
  The functors $\Tot^t$ preserve homotopy fiber sequences, so applying $\Tot^t$ to this fibration sequence evaluated on the morphism $A\to X\otimes_A \sk_n\Delta^\bullet_*$ yields a commuting diagram
{\footnotesize{ \[ \xymatrix{ \widetilde{\Gamma}_k^fF(A)\ar[d]_\simeq \ar[r] &  \Gamma_k^fF(A) \ar[r]\ar[d] &  \Gamma_{k-1}^fF(A)\ar[d]^\simeq \\ \Tot^t |\widetilde{\Gamma}_k^fF(X\underset{A}{\otimes} \sk_n \Delta_*^\bullet)| \ar[r] & \Tot^t|\Gamma_k^fF(X\underset{A}{\otimes} \sk_n \Delta_*^\bullet)| \ar[r] & \Tot^t|\Gamma_{k-1}^fF(X\underset{A}{\otimes} \sk_n\Delta_*^\bullet)|.}\] }}
\hskip -5pt {with} fibration sequences in each row.  Here we have used that $\Tot^t Z=Z$ when $Z$ is a constant cosimplicial object.  We have already shown that the left hand arrow is a weak equivalence whenever $t\geq (n+1)k$.  Assuming that the right hand arrow is a weak equivalence whenever $t\geq (n+1)(k-1)$, we can conclude that the middle arrow is a weak equivalence whenever $t\geq (n+1)k$.

\end{proof}

The pro-equivalence of Proposition \ref{p:degree_n} extends to a pro-equivalence of the discrete calculus towers associated to each functor.  We prove this next.

\begin{thm}\label{t:infinity}  Suppose that $F:\AC\rightarrow \s$ commutes with realizations.  Let $f:A\rightarrow B$ and let $X$ be any object in $\Cf$.  For any $n\geq 0$, the tower of  spectra
\[ \{V_kF(f)\}_{k\geq 1}=\{ \Gamma_k^f F(A)\}_{k\geq 1} \]
is pro-equivalent to the tower of  $\Tot$-towers of spectra
\[ \{\Tot^{m(k+1)} | \Gamma^f_k F(X\otimes_A \sk_n \Delta^\bullet_*)| \}_{k\geq 1}.\]
\end{thm}

\begin{proof}
For any $k\geq 0$ and any $m\geq n+1$, we have a commuting square
\[ \xymatrix{ \Gamma_{k+1}^f F(A) \ar[r]^{\alpha\qquad\quad\qquad} \ar[dd]_{q_{k+1}^f} & \Tot^{m(k+1)} | \Gamma^f_{k+1}F(X\otimes_A \sk_n \Delta_*^\bullet)| \ar[d]^{\Tot(q^f_{k+1})} \\
& \Tot^{m(k+1)} | \Gamma^f_k F(X\otimes_A \sk_n \Delta_*^\bullet)| \ar[d]^b \\
\Gamma_k^f F(A)\ar[r]^{\alpha'\qquad\qquad} & \Tot^{mk} | \Gamma^f_k F(X\otimes_A \sk_n\Delta_*^\bullet)|}\]
where $q_n^f$ is the natural transformation from Theorem \ref{t:tower}, and $b$ is the usual fibration between stages of the $\Tot$-tower.  The horizontal maps $\alpha$ and $\alpha'$ are the maps induced by the inclusion $\emptyset\to \sk_n\Delta_*^\bullet$ from Proposition \ref{p:degree_n}.  By the proof of Proposition \ref{p:degree_n},  $\alpha$, $\alpha'$ and $b$ are weak equivalences whenever $m\geq n+1$.  The map $\Tot(q_{k+1}^f)$ provides the diagonal map from Definition \ref{d:pro}, hence the towers are pro-equivalent.

\end{proof}

An immediate consequence of Theorem \ref{t:infinity} is that there is a weak equivalence of spectra
\[ V_\infty F(f) \to \Tot |\Gamma^f_\infty F(X\otimes_A \sk_n\Delta_*^\bullet)|\]
%(or, equivalently, ${\Gamma}^{\AC}_\infty F(f:A\to B) \simeq \Tot | \Gamma^f_\infty F(X\otimes_A \sk_n \Delta_*^\bullet)|$)
obtained by taking the inverse limit of the towers in the statement of the theorem.
We can use this result to better understand the relationship between  the limit of the  varying center tower $\{ {V_n}F\}$
and the  functor $F$.   The last two results of the paper show that  if $F$ is analytic in the sense of Definition \ref{d:rho}, then $ {V}_{\infty}F$ is equivalent to the functor that takes $A\rightarrow B$ to 
$F(B\otimes _A sk_n\Delta ^{\bullet})$.    In light of Proposition \ref{p:Tot} and since the map $A\to B\otimes_A sk_n\Delta^\bullet$ is induced by $\emptyset\to \sk_n \Delta^\bullet$,  this tells us that the  failure of the varying center tower $\{ V_nF\}$ 
to converge is measured by the failure of $F$ to commute with $\Tot$.

In the case of spaces or spectra, we use Corollary 1.4 of \cite{Rosona-AGT} to show that the  equivalence between $V_{\infty}F$ and  $F(B\otimes _A \sk_n\Delta ^{\bullet})$ holds when $F$ is a weakly $\rho$-analytic functor, as defined below.    
\begin{defn}\label{d:rho}
Let $F:\C\rightarrow \s$ where $\C$ is either $Top$ or $\s$.  Let $f:A\rightarrow B$ be a morphism in $\C$.   We say that $F$ is weakly $\rho$-analytic relative to $f$ provided that for any object 
$A\rightarrow X\rightarrow B$ in $\Cf$ where $X\rightarrow B$ is $\rho$-connected,
\[F(X)\overset{\simeq\ }{\rightarrow} \Gamma_{\infty}^fF(X).
\] 
\end{defn}
This condition is related to Goodwillie's stronger condition of $\rho$-analyticity (see \cite{G2} for details) in the sense that both conditions guarantee convergence of Taylor towers.   In particular, any $\rho$-analytic functor is also a weakly $\rho$-analytic functor.  

\begin{thm} \label{t:connected} Suppose that $F:\C \to \s$ is a homotopy functor where $\C$ is either $Top$ or $\s$, $f:A\rightarrow B$ is a $c$-connected map in $\C$, and $F$ is weakly $\rho$-analytic relative to $f$.   
If $F$ commutes with realizations, then there is  a weak equivalence of spectra
\[ V_\infty F(f) \overset{\simeq\ }{\rightarrow} \Tot |F(B\otimes_A \sk_n\Delta^\bullet )|\]
whenever $n\geq \rho-c-1$ is a non-negative integer.
\end{thm}
\begin{proof}

 In the category of spaces or spectra, the product $B\otimes_A U$ is the same as the join construction $A*_B U$ of \cite{G3} for any finite set $U$.  From \cite{G3}, we have the  useful facts that
\begin{itemize}
\item  $A*_B(U*V) \cong (A*_B U)*_B V$ where $U*V$ is the ordinary join of two spaces, and
\item $A*_B U \to B$ is at least $(m+1)$-connected if $f:A\to B$ is at least $m$-connected and $U$ is not empty.
\end{itemize}
Thus, if $f:A\to B$ is at least $c$-connected, then
\[ A*_B (\sk_0\Delta^k * \cdots * \sk_0 \Delta^k) \to B \]
is at least $(c+n+1)$-connected, where $\sk_0\Delta^k * \cdots * \sk_0\Delta^k$ is the join of $n+1$ copies of $\sk_0\Delta^k$ with itself.  Theorem 1.2 of \cite{Rosona-AGT} combined with Remark 7.1.3 of  \cite{Rosona} says that for any homotopy functor $F$, there is a weak equivalence
\[ \Tot |F(A*_B (\sk_0\Delta^\bullet * \cdots * \sk_0\Delta^\bullet))| \simeq \Tot |F(A*_B\sk_n \Delta^\bullet)|\]
where  $\sk_0\Delta^\bullet * \cdots * \sk_0\Delta^\bullet$ denotes the join of $n +1$ copies of $\sk_0\Delta^\bullet$ with itself.  Since $F$ is weakly $\rho$-analytic, for each $k$ we have a weak equivalence
\[F(A*_B (\sk_0\Delta^k * \cdots * \sk_0\Delta^k)) \simeq \Gamma^f_\infty F(A*_B (\sk_0\Delta^k* \cdots * \sk_0\Delta^k))  \]
as long as $n\geq \rho-c-1$.  This levelwise equivalence of cosimplicial spectra assembles to produce an equivalence of the associated total complexes.  Putting this together with the aforementioned result from \cite{Rosona-AGT} and \cite{Rosona}, we have a weak equivalence
\[ \Tot |F(A*_B \sk_n \Delta^\bullet)| \simeq \Tot |\Gamma^f_\infty F(A*_B \sk_n \Delta^\bullet)|.\]
The left hand side of this equivalence is $\Tot |F(B\otimes_A \sk_n\Delta^\bullet)|$, the right hand side is $\Tot |\Gamma^f_\infty F(B\otimes_A \sk_n\Delta^\bullet)|$, and the conclusion now follows from Theorem \ref{t:infinity}.

\end{proof}

As a special case of Theorem \ref{t:connected}, we obtain the desired comparison of the Goodwillie-Waldhausen and Rezk constructions.

\begin{cor}\label{c:last}
Let $\mathcal U$ be the forgetful functor from  rational commutative ring spectra to modules and let $f:{\mathbb Q}\rightarrow B$ be a morphism of rational commutative ring spectra.   Then there is a weak equivalence
\[V_{\infty}{\mathcal U}({\mathbb Q}\rightarrow B)\simeq {\rm Tot}|B\otimes _{\mathbb Q} \sk _1{\Delta}_*^{\bullet}|.
\]
\end{cor}
\begin{proof}
By the Blakers-Massey theorem (see \cite{G2}, or \cite{DS} for spectra), ${\mathcal U}$ is $1$-analytic.   Using this, the result follows immediately from Theorem 4.6.
\end{proof}

When $B=*$, the usual terminal object of $Top$ or $\s$, Theorem \ref{t:connected} follows more directly from Theorem \ref{t:infinity}.   In particular, $\sk_n \Delta^{\bullet}$ is at least $n$-connected, and hence,  $A*_B \sk_n\Delta^{\bullet}\to B$ is at least $n$-connected.    If $F$ is weakly $n$-analytic, this is sufficient to conclude that  $F(A*_B \sk_n\Delta^{\bullet}) \simeq \Gamma_\infty^fF(A*_B \sk_n\Delta^{\bullet})$ and so in this case, Theorem \ref{t:connected} follows immediately from Theorem \ref{t:infinity}. 

More generally, the condition that $F$ be weakly $\rho$-analytic can be replaced with a condition dictating that $F$ and $\Gamma_{\infty}F$ are equivalent on the objects $X\otimes _A \sk_n\Delta^k$.   
\begin{cor}  \label{p:analyticprep} Let $F:\C\rightarrow \s$ be a  functor that commutes with realizations.  Let $f:A\rightarrow B$ be a morphism in $\C$ and $X$ be an object in $\Cf$.   Suppose there exists an $n\geq 0$ for which we have a weak equivalence of spectra
\[  |F(X\otimes_A \sk_n \Delta^k)| \simeq  |\Gamma_\infty^fF(X\otimes_A \sk_n\Delta^k)| \]
for all cosimplicial degrees $k\geq 0$.  Then $\Tot |F(X\otimes_A \sk_n \Delta^\bullet)| \simeq \Gamma^f_\infty F(A)=V_{\infty}F(f:A\rightarrow B)$.
\end{cor}
\begin{proof}  The levelwise hypothesis of the statement guarantees that there is a weak equivalence of spectra  \[\Tot| F(X\otimes_A \sk_n\Delta^\bullet)| \simeq \Tot |\Gamma^f_\infty F(X\otimes_A \sk_n \Delta^\bullet)|.\]   Composing with the equivalence from Theorem \ref{t:infinity}, we have 
\[\Gamma_\infty^fF(A) \simeq \Tot |F(X\otimes_A \sk_n \Delta^\bullet)|,\]
which implies the result.

\end{proof}

The key point in requiring the existence of $n$ in Corollary \ref{p:analyticprep} is that the space $\sk_n\Delta^k$ is at least $n$-connected for all $k$.    So, like the analyticity condition in Theorem \ref{t:connected}, the condition in this proposition requires convergence on analogues of $n$-connected objects.

%---------------------------------

\end{document}